\documentclass[11pt,twoside]{amsart}
\usepackage{amssymb}
\usepackage[matrix,arrow]{xy}
\usepackage{graphics}
\usepackage{palatino}
\usepackage{pgf}

\theoremstyle{plain}
\newtheorem{prop}{Proposition}[section]
\newtheorem{lemma}[prop]{Lemma}
\newtheorem{thm}[prop]{Theorem}
\newtheorem{cor}[prop]{Corollary}
\theoremstyle{remark}
\newtheorem{rmk}[prop]{Remark}
\theoremstyle{definition}
\newtheorem{defn}[prop]{Definition}
\newtheorem{ex}[prop]{Example}

\renewcommand{\AA}{\mathbb{A}}

\newcommand{\CC}{\mathbb{C}}

\newcommand{\FF}{\mathbb{F}}

\newcommand{\PP}{\mathbb{P}}
\newcommand{\QQ}{\mathbb{Q}}
\newcommand{\RR}{\mathbb{R}}

\newcommand{\ZZ}{\mathbb{Z}}

\newcommand{\defined}{\mathrel{\mathop:}=}
\newcommand{\Sp}{\operatorname{Sp}}
\newcommand{\MGL}{\mathsf{MGL}}
\newcommand{\BP}{\mathsf{BP}}
\newcommand{\BPn}[1]{\mathsf{BP}\langle {#1} \rangle}
\newcommand{\kgl}{\mathsf{kgl}}
\newcommand{\KGL}{\mathsf{KGL}}
\newcommand{\SH}{\mathsf{SH}}
\newcommand{\ku}{\mathsf{ku}}
\newcommand{\ko}{\mathsf{ko}}
\newcommand{\EEE}{\mathsf{E}}
\newcommand{\M}{\mathsf{M}}
\newcommand{\HHH}{\mathsf{H}}
\newcommand{\one}{\mathbf{1}}
\newcommand{\MU}{\mathsf{MU}}
\newcommand{\x}{\mathsf{x}}
\newcommand{\y}{\mathsf{y}}

\newcommand{\cd}{\text{cd}}
\newcommand{\smsh}{\wedge}
\newcommand{\Ext}{\operatorname{Ext}}
\newcommand{\cotensor}{\mathbin{\Box}}
\newcommand{\Spec}{\operatorname{Spec}}

\allowdisplaybreaks[1]

\begin{document}

\title{Motivic Brown-Peterson invariants of the rationals}
\author{Kyle M.~Ormsby}
\address{Department of Mathematics, MIT, USA}
\email{ormsby@math.mit.edu}

\author{Paul Arne {\O}stv{\ae}r}
\address{Department of Mathematics, University of Oslo, Norway}
\curraddr{Department of Mathematics, MIT, USA}
\email{paularne@math.uio.no}

\begin{abstract}
Let
  $\BPn{n}$, $0\le n\le \infty$, denote the family of motivic
  truncated Brown-Peterson spectra over $\mathbb{Q}$.  We employ a
  ``local-to-global" philosophy in order to compute the bigraded homotopy groups of
  $\BPn{n}$.  Along the way, we produce a computation of the homotopy
  groups of $\BPn{n}$ over $\QQ_2$, prove a motivic Hasse principle for
  the spectra $\BPn{n}$, and reprove several classical and recent
  theorems about the $K$-theory of particular fields in a streamlined fashion.  We also compute the bigraded homotopy groups of the 2-complete algebraic cobordism spectrum $\MGL$ over $\QQ$.
\end{abstract}

\maketitle

\tableofcontents

\section{Introduction}\label{sec:intro}
In \cite{MorelVoevodsky}, Morel and Voevodsky initated a new approach
to studying arithmetic questions by homotopy theoretic means. This involved
defining, for any field $k$, an entire homotopy theory called ``motivic homotopy
theory over $k$." As part of the apparatus
of homotopy theory, one is entitled to a category of objects stabilized by
inverting the operation of smashing with a fixed object (such as $S^1$, in the
classical case). In the motivic case, the most arithmetically interesting
choice, of several available, is to localize by inverting smashing with the
projective line $\PP^1$.  The outcome of this procedure is Voevodsky's
stable motivic homotopy category, $\SH(k)$ \cite{Voevodskystable}; we
call its objects motivic spectra.  Many important
arithmetic and algebro-geometric cohomology theories are represented
by motivic spectra, including motivic cohomology and motivic Steenrod
operations, algebraic $K$-theory, and algebraic cobordism.  Working in
$\SH(k)$ allows us to apply numerous techniques from computational
algebraic topology to these invariants.  Primary among these examples
is the resolution of the Milnor conjecture on Galois cohomology
\cite{Voevodskymod2}.  There has been a recent flurry of concrete computations in
stable motivic homotopy theory, led by the efforts of Hu-Kriz-Ormsby,
Dugger-Isaksen, and Hill \cite{HKO,Ormsby,DI,Hill}.

Each motivic spectrum $\EEE$ has bigraded homotopy groups
\[
 \pi_\star \EEE = \bigoplus_{m,n\in \ZZ}\pi_{m+n\alpha}\EEE
\]
where $\pi_{m+n\alpha}\EEE$ consists of stable homotopy classes of maps $S^{m+n\alpha}\to \EEE$.  Here $S^{m+n\alpha} = (S^1)^{\wedge m} \wedge (\AA^1\smallsetminus 0)^{\wedge n}$.

For judicious choices of $\EEE$, these bigraded groups carry important
information about the base field.  For instance, if $\EEE = \KGL$ is
the motivic algebraic $K$-theory spectrum, then $\pi_{m+0\alpha}\KGL =
K_m(k)$, the $m$-th algebraic $K$-theory group of $k$.  If $\EEE =
\MGL$ is the algebraic coboridsm spectrum, then the groups
$\pi_{*(1+\alpha)}\MGL = \bigoplus_{k\in \ZZ} \pi_{k(1+\alpha)}\MGL$
form a ring which corepresents one-dimensional commutative formal
groups laws (the Lazard ring).

Let $\BP = \BP_k$ denote the motivic Brown-Peterson spectrum at the
prime $2$ over a characterstic 0 base field $k$.  (\emph{N.B.} We will
often drop the subscript $k$ from our motivic spectrum notation if it
is clear that only one field is in play, but we will put the subcript
back in whenever the base field needs to be clarified.)  The spectrum $\BP$ was
constructed by Po Hu and Igor Kriz \cite{HuKriz} and Gabriele Vezzosi
\cite{Vezzosi} on the model of Dan Quillen's idempotent construction in
topology \cite{Quillen} and serves as the universal $2$-typical component of
$\MGL$.  
The motivic truncated Brown-Peterson spectra $\BPn{n} = \BPn{n}_k$, $0\le n\le
\infty$, comprise a tower of $\BP$-modules
\begin{equation*}
\BP = \BPn{\infty}
\rightarrow
\cdots
\rightarrow
\BPn{n}
\rightarrow
\BPn{n-1}
\rightarrow
\cdots
\rightarrow
\BPn{0}.
\end{equation*}
Here,
$\BPn{0}=\M\ZZ_{(2)}$ is the $2$-local motivic Eilenberg-MacLane
spectrum by a theorem of Hopkins and Morel, 
while $\BPn{1}$ is the $2$-local connective 
$K$-theory spectrum and $\BP$ is the universal $2$-typical
algebraically oriented spectrum.  (See Lemma
\ref{lem:kgl} for the precise way in which $\BPn{1}$ is related to $\KGL$.) As such, for
$n>1$ we can view the groups $\pi_\star \BPn{n}$ as higher height generalizations of
the algebraic $K$-theory of the base field that are constructed from
algebraic cobordism.  

Our central result is a computation of the bigraded homotopy groups of
2-complete $\BPn{n}$ over the base field $\QQ$ via the motivic Adams 
spectral sequence.  The $n=1$ and $n=\infty$ case are of special interest.  In the $n=1$ case, we garner a new computation of the 2-complete algebraic $K$-theory of $\QQ$ originally arrived at by Rognes and Weibel via the Bloch-Lichtenbaum spectral sequence \cite{RW}.   Our computation of $\pi_\star\BP$
is a first step in a program for computing the motivic stable homotopy
groups of the sphere spectrum over $\QQ$ via the motivic Adams-Novikov
spectral sequence.

In topology, Milnor and Novikov showed independently that complex cobordism
splits as a wedge of suspensions of the Brown-Peterson spectrum.  Via
the motivic Adams spectral sequence, we show in Theorem
\ref{thm:wedge} that whenever the $2$-cohomological dimension of
$k(\sqrt{-1})$ is finite, $\MGL_k$ is the expected wedge of
suspensions of $\BP_k$'s.  (Our condition on $\operatorname{cd}_2k(\sqrt{-1})$
ensures that the motivic Adams spectral sequence converges.)  As such,
Theorem \ref{thm:BPnQHtpy} also produces a computation of the bigraded
coefficients of $\MGL_\QQ$.  This greatly extends the work of
Naumann-Spitzweck-{\O}stv{\ae}r \cite{NSO2} which computes
\[
  \pi_{m+0\alpha}\MGL\otimes \QQ =
  \begin{cases}
    \QQ &\text{if }m=0\text{ or }m>0\text{ and }4\mid m-1,\\
    0 &\text{otherwise}.
  \end{cases}
\]

For the topological $\BPn{n}$ the Adams spectral sequence collapses at the $E_{2}$-page, 
and its homotopy is 
\begin{equation*}
\pi_*\BPn{n}
=
\ZZ_{(2)}[v_{1},\cdots,v_{n}],
\vert v_{i}\vert = 2^{i+1}-2.
\end{equation*}
Over $\QQ$ the motivic Adams spectral sequence for $\BPn{n}$ does not collapse at some finite page.
We display an elaborate pattern of differentials governed by a Hasse principle in motivic homotopy theory,
which is reminiscent of the ``local-to-global" methods employed with
much success in class field theory and the study of quadratic forms.

Indeed, for each real or $p$-adic completion $\QQ_v$ of the rationals $\QQ$ there is a canonical map
\[
  \pi_\star\BPn{n}_{\QQ}\to \pi_\star\BPn{n}_{\QQ_v}
\]
to which a map of motivic Adams spectral sequences converges
(cf.~Proposition \ref{prop:HasseOnHtpy}).  In Theorem \ref{thm:Hasse}
we prove that the product of these maps is an injection.  The theorem
depends on a partial analysis of the motivic Adams spectral sequence
for $\BPn{n}_\QQ$ and then permits a full computation of the spectral
sequence (Theorems \ref{thm:BPnQ} and \ref{thm:BPnQHtpy}).

A more thorough outline of our method is as follows.  First, for each
real or $p$-adic completion of $\QQ$ and for
$\QQ$ itself, we run a Bockstein spectral sequence which converges to
the $E_2$-term of the motivic Adams spectral sequence (MASS) for
$\BPn{n}$.  (These Bockstein spectral sequences are based on filtering the
dual Steenrod algebra by powers of $\rho$, the class of $-1$ in Milnor
$K$-theory.)  We then run the MASS for each real or $p$-adic
completion.  (These computations are already known for $\RR$ and
$\QQ_p$, $p>2$, while our method in this paper is the first to produce
a $\BPn{n}$ computation over $\QQ_2$ for any $n>0$.)  Finally, all
these computations are combined and analyzed with respect to a
global-to-local map that allows us to compute the MASS for $\BPn{n}$
over $\QQ$. \footnote{An alternate title for this paper is \emph{How to compute
    the motivic Brown-Peterson homology of $\QQ$ in $1+(2+5\cdot
    \infty)\cdot \infty$ easy steps}.  We leave it as an exercise to
  the reader to derive this joke by computing the number of spectral
  sequence pages with nontrivial differentials used in our argument.}

Naturally, the results of these computations are quite complicated and
are hard to understand without a thorough familiarity with the spectral
sequences.  We have made a significant effort to present these
computations in as digestible a format as possible.  The reader can
find diagrams for the $\rho$-BSS and MASS for $\BPn{3}_\QQ$ in Figures
\ref{fig:A}, \ref{fig:B}, and \ref{fig:C}.  Enough of the patterns
present in such computations figure in the $n=3$ case that the
reader --- with enough diligence and patience --- should be able to
produce such diagrams for abitrary $n$.  The reader is warned that
these are not standard spectral sequence charts (homological degree is
suppressed and only certain $v_i$-multiplications are drawn
explicitly) but the graphical calculus is explained in Remarks
\ref{rmk:graph1} and \ref{rmk:graph2}.

A closed form for these groups (for arbitrary $n$) is presented in
Theorem \ref{thm:BPnQHtpy} but some qualitative remarks are in order
here; we focus on the $m+0\alpha$, $m\ge 2$, component of the bigraded homotopy
groups just to give a flavor of the answers (see Section
\ref{sec:MASS} for an explanation of our grading conventions).  First,
when $m$ is even these groups contain an infinitely generated direct
sum of cyclic 2-groups of unbounded order.  This summand
depends on $m$ but is 
independent of $n$.  Depending on $m$ and $n$, a certain finite number of
$\ZZ/2$ summands may also appear.  If $m\equiv 1\pod{4}$, then
$\pi_m\BPn{n}$ contains a $\ZZ_2$ summand, and this describes all of the
non-torsion; an $m,n$-dependent finite 
number of $\ZZ/2$ summands may also appear in these degrees.  Finally, if $m\equiv
3\pod{4}$ we find an $m,n$-dependent finite $2$-torsion group.

There is a tantalizing connection between the groups calculated in
Theorem \ref{thm:BPnQHtpy} and the standard localization sequence
\[
  \bigvee_p K\FF_p \to K\ZZ \to K\QQ
\]
in algebraic $K$-theory.  In fact, in the $n=1$ case our computation
naturally splits into components abstractly accounting for the contribution of
$\bigvee_p K\FF_p$ and $K\ZZ$ to $K\QQ$.  There is a similar abstract splitting for $n>1$ in which case there are no classical localization theorems for
$\BPn{n}$ (or associated spectra like motivic $\EEE(n)$).  This leads us
to speculate that the motivic spectra $\BPn{n}$, $0\le n\le \infty$,
should satisfy some sort of localization property (although there are
technical details that make the precise statement of such a conjecture
nontrivial).  We explore these ideas in Remark
\ref{rmk:loc}; they should provide the basis for continued research
on the $\BPn{n}$ spectra.

We now indicate the precise outline of our paper:

In Section \ref{sec:MASS} we review the motivic Adams spectral sequence (denoted by MASS), 
the construction of $\BPn{n}$, 
and the comodule structure of the motivic homology of $\BPn{n}$ over the dual Steenrod algebra.
We recall that the MASS converges for $\BPn{n}$ over fields of finite virtual mod-$2$ \'etale cohomological dimension,
a condition which holds for $\QQ$ and all of its completions.

In Section \ref{sec:comp} we review known MASS computations for $\BPn{n}$ over the real numbers and the 
$p$-adic rationals, $p>2$,
along with a number of applications. 
We then compute the groups $\pi_\star\BPn{n}_{\QQ_2}$.
The usefulness of these computations will be evident in the last part of the paper.

In Section \ref{sec:Hasse} we use base change functors to construct ``rational models'' for motivic spectra.  
We apply this to truncated Brown-Peterson spectra.
The unit of a base change adjunction allows us to construct the ``Hasse map'' that compares spectra defined 
over global and local number fields.

Finally, 
in Section \ref{sec:rational}, 
we combine the Hasse map with the $p$-adic and real MASS computations to prove the Hasse principle for $\BPn{n}$ over $\QQ$, 
and determine its coefficients.

\textbf{Acknowledgments.}  Both authors would like to thank Mike Hill
for input on this project during the summer of 2009; they would also
like to thank Haynes Miller
for helpful comments during the preparation of this manuscript.

The first author would like to thank Igor Kriz for initially suggesting that the local-to-global philosophy might be useful in motivic homotopy theory; he also acknowledges partial support from NSF award DMS-1103873.

The second author would like to thank the MIT Mathematics Department for
its hospitality and acknowledges partial support from the Leiv
Eriksson mobility programme and RCN ES479962.

Finally, both authors would like to thank the referee for timely and
constructive comments, and the editors for helpful improvements to the
exposition.

\section{MASS for $\BPn{n}$}\label{sec:MASS}

\subsection{MASS}

For $\EEE$ a motivic spectrum let $\pi_\star \EEE  = \EEE_\star$
denote the bi-graded coefficients
\[
  \bigoplus_{m,n\in \ZZ} \pi_{m+n\alpha} \EEE = \bigoplus_{m,n\in \ZZ} [S^{m+n\alpha},\EEE]
\]
where $S^{m+n\alpha} = (S^1)^{\wedge m} \wedge (\AA^1\smallsetminus
0)^{\wedge n}$.

Let $\M\ZZ$ be the integral motivic Eilenberg-MacLane spectrum. 
Its mod-$2$ version $\M\ZZ/2$ is defined as the smash product of $\M\ZZ$ with the mod-$2$ motivic Moore spectrum $\one/2$.
The bigraded homotopy groups of $\M\ZZ/2 \wedge \M\ZZ/2$ in the motivic stable homotopy category over any field of characteristic zero 
identifies with the dual motivic Steenrod algebra \cite{Voevodskyreduced}.
\begin{prop}
Over fields of characteristic zero the dual Steenrod algebra $\mathcal{A}_{\star}$ at the prime $2$ is isomorphic to $\M\ZZ/2_{\star}\M\ZZ/2$.
\end{prop}
We refer to \cite[Proposition 7.2]{DI} for a short proof of this result, 
which is based on the identification of Voevodsky's big category of motives with $\M\ZZ$-modules 
\cite{ROCRASmodules}, \cite{ROAdvancesmodules}, 
and the description of proper Tate motives in \cite{VoevodskyEMspaces}.
For algebraically closed fields, 
an alternate proof is given in \cite[Theorem 4]{HKO}.
(The proofs in \cite{DI} and \cite{HKO} carry over verbatim to odd primes.)

In the rest of the paper we let $\M$ denote $\M\ZZ/2$.  Let $k^M_*$
denote the mod 2 Milnor $K$-theory of the base field and recall that
$M_\star \cong k^M_*[\tau]$ where $|k^M_n| = -n\alpha$ and $|\tau| =
1-\alpha$.

Next we recall the structure of the dual Steenrod algebra $\mathcal{A}_{\star}$ at $2$ as a Hopf algebroid over the ground ring $\M_\star$
\cite{Voevodskymod2}, \cite{Voevodskyreduced}.
Throughout we use the standard grading convention 
\[
\M_{\star}=\M^{-\star}.
\]
To begin with,
\[
\mathcal{A}_{\star}
=
(\M_{\star},
\M_{\star}[\xi_{1},\dots][\tau_{0},\dots]/(\tau_{i}^{2}-\rho(\tau_{i+1}-\tau_{0}\xi_{i+1})-\tau\xi_{i+1})).
\]
The left unit in the Hopf algebroid structure is the canonical inclusion, 
while the right unit is determined by 
\[
\eta_{R}(\rho)=\rho, \eta_{R}(\tau)=\tau+\rho\tau_{0},
\]
for the canonical classes $\tau\in\M_{1-\alpha}\cong\mu_{2}(k)$ and $\rho\in\M_{-\alpha}\cong k^{\times}/(k^{\times})^{2}$.
The mod $2$ Bockstein on $\tau$ equals $\rho$.
We note that $\M_{\star}\M$ is a commutative free $\M_{\star}$-algebra.
Moreover, 
the polynomial generators have bidegrees 
\[
\vert\xi_{i}\vert=(2^i-1)(1+\alpha),\, 
\vert\tau_{i}\vert=1+(2^i-1)(1+\alpha),
\] 
and coproducts given by 
\[
\Delta\xi_{i} =\sum_{j=0}^{i}\xi_{i-j}^{2^{j}}\otimes\xi_{j},\,
\Delta\tau_{i}=\tau_{i}\otimes 1+\sum_{j=0}^{i}\xi_{i-j}^{2^{j}}\otimes\tau_{j}.
\] 
These are the same formulae as in topology \cite{MilnorSteenrodAlg}, \cite[Theorem 3.1.1]{Ravenelbook}.
While $\tau$ is not primitive in general,
the graded mod-$2$ Milnor $K$-theory ring $k_{\ast}^{M}\subseteq\M_{\star}$
of the base field comprises primitive elements.

\begin{rmk}
Details on the odd-primary dual motivic Steenrod algebra 
in \cite{Voevodskyreduced} will not be recounted here since all of our
computations occur with $p=2$.
\end{rmk}

Suppose $\EEE$ is a motivic homotopy ring spectrum, 
i.e., a ring object in the motivic stable homotopy category.
The homotopy fiber sequence
\[
\overline{\M}
\rightarrow
\one
\rightarrow
\M
\]
gives rise to the Adams resolution
\begin{equation}
\label{adamstower}
\xymatrix{
\cdots\ar[r] & \EEE_{s+1} \ar[r] & \EEE_{s} \ar[r]\ar[d] & \EEE_{s-1} \ar[r] & \cdots \\
& & \M\smsh\overline{\M}^{\smsh (s)}\smsh\EEE }
\end{equation}
where 
\[
\EEE_{s}=\overline{\M}^{\smsh (s)}\smsh\EEE.
\]
The K{\"u}nneth isomorphism for motivic cohomology \cite[Proposition 7.5]{DI} and standard arguments,
cf.~\cite{DI}, \cite{HKO}, 
show that the homotopy spectral sequence associated to (\ref{adamstower}) is a conditionally convergent spectral sequence
\begin{equation}
\label{EASS}
E_{2}^{s,m+n\alpha}=
\Ext_{\mathcal{A}_{\star}}^{s,m+n\alpha}(\M_{\star},\M_{\star}\EEE)
\Longrightarrow
\pi_{m-s+n\alpha}\EEE\,\widehat{}.
\end{equation}
The target graded group is the motivic homotopy $\pi_{\star}\EEE\,\widehat{}$ of the nilpotent $\M$-completion of $\EEE$.  
This is a tri-graded spectral sequence, 
where $s$ is the homological degree of the Ext group (the Adams filtration), 
$m+n\alpha$ is the internal motivic bigrading coming from the bigrading on $\mathcal{A}_{\star}$ and $\M_{\star}$. 

The problem of strong convergence of (\ref{EASS}) is discussed in \cite{KO}.
Recall that $\EEE$ is of finite type if $\EEE_{m+n\alpha}=0$ for $m\ll 0$.
In all of the examples in this paper, 
the coefficient ring vanishes for $m<0$.
\begin{thm}
\label{thm:ASSconvergence}
Suppose $\EEE$ is cellular and of finite type and $\cd_{2}(k(\sqrt{-1}))<\infty$.
Then the $\M$ based Adams spectral sequence (\ref{EASS}) is strongly convergent
to the homotopy groups of the Bousfield localization of $\EEE$ at $\mathbf{1}/2$.
\end{thm}

\begin{ex}
The assumptions in Theorem \ref{thm:ASSconvergence} hold when $k=\QQ$, $\RR$ and $\QQ_{p}$ and $\EEE$ 
is one of the following motivic spectra.
\begin{itemize}
\item The sphere spectrum $\one$. 
It follows by results in \cite{Morelstableconnectivity} that $\pi_{m+n\alpha}\one=0$ for $m<0$, cf.~\cite{Morelintroduction}.
\item Algebraic cobordism $\MGL$.
We have $\MGL_{m+n\alpha}=0$ for $m<0$ and $\MGL_{-n\alpha}=K_{n}^{M}$ (Milnor $K$-groups of the base field) for $n\ge 0$ \cite{Voevodskystable}.
\item The algebraic Brown-Peterson spectrum $\BP$.
See \cite{HuKriz, Vezzosi} for constructions.  The coefficient ring vanishes for $m<0$ according to the previous example.
\item The motivic truncated Brown-Peterson spectra $\BPn{n}$.\footnote{In a few instances, $\BPn{n}$ will refer to the topological truncated $\BP$; this will always be clear from context.}  Again the coefficients vanish for $m<0$ by the previous example.
\end{itemize}
\end{ex}

\subsection{$\BPn{n}$}

We use the MASS to determine the coefficients of the truncated motivic Brown-Peterson spectra $\BPn{n}$.  
Here we review the definition and homology of $\BPn{n}$, 
and specify how we use the latter to compute the $E_2$-page of the MASS.  
We also recall the identifications of $\BPn{0}$ and $\BPn{1}$ in terms of familiar motivic spectra.

Let $\BP = \BPn{\infty}$ denote the motivic Brown-Peterson spectrum constructed from $2$-local algebraic 
cobordism $\MGL_{(2)}$ via the Quillen idempotent \cite{HuKriz,Vezzosi}.  
Inside $\MGL_{(2)\star}$ and $\BP_\star$ there are the usual classes $v_i$ of degree $(2^i-1)(1+\alpha)$.
The elements $v_{n+1},v_{n+2},\ldots$ comprise a regular sequence.
Following the script in topology, 
the motivic $\BPn{n}$ is the $\BP$-module formed by killing off the regular sequence
$v_{n+1},v_{n+2},\ldots$ in $\BP_\star$.

In order to understand the homology of $\BPn{n}$ as a comodule over
$\mathcal{A}_\star$ we introduce auxiliary Hopf algebroids
$\mathcal{E}(n)$ from \cite{Hill}.

\begin{defn}\label{defn:En}
Let $\mathcal{E}(n)$ denote the quotient Hopf algebroid
\[\begin{aligned}
  \mathcal{E}(n) &\defined (\M_\star,
  \mathcal{A}/(\xi_1,\xi_2,\ldots)+(\tau_{n+1},\tau_{n+2},\ldots))\\
  &= (\M_\star,\M_\star[\tau_0,\ldots,\tau_n]/(\tau_i^2-\rho\tau_{i+1},\tau_n^2)).
\end{aligned}\]
We permit $n=\infty$, in which case
\[\begin{aligned}
  \mathcal{E}(\infty) &\defined
  (\M_\star,\mathcal{A}/(\xi_1,\xi_2,\ldots))\\
  &= (\M_\star,\M_\star[\tau_0,\tau_1,\ldots]/(\tau_i^2 - \rho\tau_{i+1})).
\end{aligned}\]
\end{defn}

In \cite{Ormsby}, the first author determines the homology of $\BPn{n}$ as a comodule
over $\mathcal{A}_\star$.

\begin{thm}\label{thm:HBPn}
For $0\le n \le \infty$,
there is an isomorphism of Hopf algebroids
\[
  \M_\star \BPn{n} = \mathcal{A}\cotensor_{\mathcal{E}(n)} \M_\star.
\]
\end{thm}

By change-of-rings we can rewrite the $E_2$-term of the MASS for
$\BPn{n}$.

\begin{thm}\label{thm:E2BPn}
For $0\le n \le \infty$, the $E_2$-term of the MASS for $\BPn{n}$ is
isomorphic to
\[
  \Ext_{\mathcal{E}(n)}(\M_\star,\M_\star).
\]
\end{thm}

Theorem \ref{thm:E2BPn} provides computational control over the
$E_2$-term of the MASS, and in the next section we will review how the
$\rho$-Bockstein spectral sequence produces explicit calculations over
particular fields.

Currently, we make precise the connections between $\BPn{0}$ and
$\BPn{1}$ and more well-known motivic spectra.  These provide motivation
for thinking of the $\BPn{n}$, $n\ge 2$, as higher chromatic level
spectra in the motivic context (in a sense which we will not make
precise in this paper).  Throughout our computations we will use
the connection between $\BPn{0}$ and $\M\ZZ$ to initiate calculations,
and for
each of our $n\ge 1$ results we will comment on the algebraic
$K$-theory implications of the $n=1$ case.

The following result follows from announced work of Hopkins and Morel.
A detailed proof is given by Hoyois in \cite{Hoyois}.
\begin{thm}
The motivic spectra $\BPn{0}$ and $\M\ZZ_{(p)}$ are isomorphic over any field of characteristic zero.
\end{thm}

The $2$-local connective algebraic $K$-theory spectrum $\kgl_{(2)}$ is precisely $\BPn{1}$ by definition.  
(At odd primes, 
$\BPn{1}$ is an Adams summand of localized connective algebraic $K$-theory \cite[\S 4]{NSO3}.)

\begin{lemma}
\label{lem:kgl}
Suppose $S$ is a separated Noetherian scheme of finite Krull dimension.
There exists a connective algebraic $K$-theory motivic spectrum $\kgl$ such that the natural map $\kgl\rightarrow\KGL$ to 
algebraic $K$-theory becomes a weak equivalence after inverting the
Bott map.  In particular, if $S= \Spec k$ we have
\[
  \pi_m v_1^{-1}\BPn{1}_k \cong K_m(k)
\]
where $K_*(k)$ denotes 2-complete algebraic $K$-theory of $k$.
\end{lemma}
\begin{proof}
Recall that $\KGL$ is the motivic Landweber exact spectrum associated to the $\MU_{\ast}$-algebra 
\begin{equation*}
x_{1}^{-1}\MU_{\ast}/(x_{2},x_{3},\dots)\MU_{\ast}\cong\ZZ[x_1,x_1^{-1}] 
\end{equation*}
classifying the multiplicative formal group law $\x+\y-x_1\x\y$, 
cf.~\cite{NSO1,NSO2}.
Here we employ a fixed isomorphism 
\[
\MU_{\ast}\cong\ZZ[x_{1},x_{2},x_{3},\dots]
\] 
of graded rings where $\vert x_{i}\vert=i$ (that is, half of the usual topological grading).
The canonical map $\MU_{\ast}\rightarrow\MGL_{*(1+\alpha)}$ affords forming the quotient 
\begin{equation*}
\kgl=\MGL/(x_{2},x_{3},\dots)
\end{equation*}
by taking iterated cofibers of the multiplication by $x_{i}\in\MGL_{i(1+\alpha)}$ map in the homotopy category of $\MGL$-modules.
The orientation map for $\KGL$ sends $x_{i}$ to $0\in\KGL_{i(1+\alpha)}$ for $i\geq 2$.
Hence there exists a naturally induced map $\kgl\rightarrow\KGL$.
In order to show this map becomes a weak equivalence when inverting the Bott map,
i.e.,
\[
x_1^{-1}\kgl\cong\KGL,
\] 
it suffices, 
by passing to the colimit, 
to show that 
\begin{equation*}
x_{1}^{-1}\MGL/(x_{2},x_{3},\dots,x_{n})
\end{equation*}
is the motivic Landweber exact spectrum associated to the $\MU_{\ast}$-module
\begin{equation*}
x_{1}^{-1}\MU_{\ast}/(x_{2},x_{3},\dots,x_{n})
\end{equation*}
for every $n\geq 2$.
This can be verified inductively, 
cf.~\cite[Theorem 5.2]{Spitzweckslices}.
\end{proof}

Later, we will use Lemma \ref{lem:kgl} and computations of $\pi_\star\BPn{1}$ to prove statements about classical algebraic $K$-theory.  One simply inverts $v_1$ and reads off the weight 0 component to determine the algebraic $K$-theory of the base field.

We conclude this section by identifying the algebraic cobordism spectrum $\MGL$ with a wedge of suspensions of $\BP$.

\begin{thm}\label{thm:wedge}
Suppose $k$ is a field with finite virtual cohomological dimension at 2.  Let $x_i$, $i\ge 1$ denote the standard Lazard ring polynomial generators in degree $i(1+\alpha)$.  Let $\BP = \BP_k$ and $\MGL = \MGL_k$ denote the 2-complete Brown-Peterson and algebraic cobordism spectra over $k$.  Let $S$ denote the set of monomials $x_I$ in the $x_i$ where no factor is of the form $x_{2^j-1}$, $j\ge 1$.  Then there is an equivalence
\begin{equation}\label{eqn:wedge}
  \bigvee_{x_I\in S} \Sigma^{|x_I|}\BP\to \MGL.
\end{equation}
\end{thm}
\begin{proof}
The $x_i$ exist because $\MGL$ is the universal algebraically oriented spectrum.  The maps in (\ref{eqn:wedge}) are given by multiplication by $x_I$.  The motivic homology of $\BP$ and $\MGL$ is known by Theorem \ref{thm:HBPn} and \cite[Proposition 6]{Bor}.  Applying the MASS to (\ref{eqn:wedge}) we get an isomorphism on $E_2$-terms.  The MASS converges to homotopy groups of 2-completions because of our hypotheses on $k$ \cite{KO}. It follows that (\ref{eqn:wedge}) is an isomorphism on homotopy groups.  Since $\MGL$ and $\BP$ are cellular, we get that (\ref{eqn:wedge}) is an equivalence.
\end{proof}

\section{Computations over completions of $\QQ$}
\label{sec:comp}

In this section we review known MASS computations of $\BPn{n}\hat{_2}_\star$
over $\CC$, $\RR$, and $\QQ_p$, $p>2$, and present a new calculation of
$\BPn{n}\hat{_2}_\star$ over $\QQ_2$.  Here $(~)\hat{_2}$ denotes
Bousfield localization at the motivic mod 2 Moore spectrum.
The differential $d_{r}$ takes the form $E_{r}^{s,m+n\alpha}\rightarrow E_{r}^{s+r,m+r-1+n\alpha}$.
When depicting MASS we shall employ ``Adams grading" by placing elements of 
$E_{r}^{s,m+n\alpha}$ in tri-degree $(m-s+n\alpha,s)$,
with $\alpha$ along the vertical axis.
Thus, 
in Adams grading, 
the $r$-th differential has tri-degree $(-1+0\alpha,r)$.
The same convention applies to Bockstein spectral sequences.

\emph{Notice:} In the rest of the paper we elide the 2-completion symbol $(~)\hat{_2}$ for legibility.  
In other words, 
we proceed to work in the $2$-complete stable motivic homotopy category.

The results over $\CC$ are due to
Hu-Kriz-Ormsby \cite{HKO}, over $\RR$ Mike Hill \cite{Hill}, and over
$\QQ_p$ Ormsby \cite{Ormsby}.  We recall the differentials here
because we will need them in Section \ref{sec:rational} to carry out
computations over $\QQ$.

\subsection{The complex place}
We begin by discussing the base field $\CC$, the complex numbers.  These results are not integral to the rest of the paper, but they serve as a nice warm-up case to familiarize the reader with our methods.  
For $\CC$ the motivic Hopf algebra $\mathcal{E}(n)$ (left and right units agree) is the base change of the topological Hopf algebra 
$\mathcal{E}(n)^{\top}$ to the mod $2$ cohomology ring $\FF_2[\tau]$ of a point.
Here $\vert\tau\vert=1-\alpha$.
Thus the following result is immediate,
cf.~\cite[Theorem 3.1.16]{Ravenelbook}.

\begin{prop}
\label{prop:ExtE1C}
Over $\CC$ there is an algebra isomorphism
\[
\Ext^{\ast}_{\mathcal{E}(n)}(\M_{\star},\M_{\star})
= \M_\star[v_0,\ldots,v_n] = 
\FF_2[\tau,v_0,\ldots,v_n]
\]
where $\vert\tau\vert=(1-\alpha,0)$ and $\vert v_{i}\vert =
((2^{i}-1)(1+\alpha),1)$ in Adams tri-grading.
\end{prop}

The generator $v_{i}$ is represented by the class of $\tau_{i}$ in the cobar complex.

\begin{thm}\label{thm:kC}
The motivic Adams spectral sequence for $\BPn{n}$ collapses at $E_2$ and
\[
\BPn{n}_\star = (\M\ZZ_2)_\star[v_1,\ldots,v_n] = \ZZ_2[\tau,v_1,\ldots,v_n]
\]
where $\vert\tau\vert=1-\alpha, \vert v_{i}\vert=(2^i-1)(1+\alpha)$.
The polynomial generator $v_{1}$ is the Bott periodicity operator for
$\kgl = \BPn{1}$.
\end{thm}

\begin{proof}
The collapse of the spectral sequence at its $E_{2}$-page follows for tri-degree reasons from Proposition \ref{prop:ExtE1C}.  
The fact that $v_{0}$ detects multiplication by $2$ in $\pi_0\BPn{n} = \ZZ_2$ resolves all multiplicative extension problems \cite[Lemma 5.4]{Hill}.
\end{proof}

In the $n=1$ case we can deduce an important fact about the algebraic $K$-theory of $\CC$ due to Suslin \cite{Suslinlocalfields}.

\begin{lemma}\label{lem:33}
Let $\mathcal{E}(n)^{\top}$ denote the variant of $\mathcal{E}(n)$ from topology.  The complex topological realization functor induces a map between the Adams spectral sequences for $\BPn{n}_\CC$
\[
\Ext_{\mathcal{E}(n)}^{s,m+n\alpha}(\M_{\star},\M_{\star})
\Longrightarrow
\pi_{m-s+n\alpha}\BPn{n}_\CC
\]
and the topological $\BPn{n}$ spectrum 
\[
\Ext_{\mathcal{E}(n)^{\top}}^{s,t}(\HHH\ZZ/2_{\ast},\HHH\ZZ/2_{\ast})
\Longrightarrow
\pi_{t-s}\BPn{n}.
\]
It sends $\tau$ to $1$ and $v_{i}$ to $v_{i}$ for $i=0,\cdots,n$, and the induced map in weight zero
\[
\pi_{m+0\alpha}\BPn{n}_\CC
\rightarrow
\pi_{m}\BPn{n}
\]
is an isomorphism for all $m\in\ZZ$.
\end{lemma}
\begin{rmk}
The isomorphism $K_{\ast}(\CC)\cong\pi_{\ast}\ku$ was shown by Suslin in \cite{Suslinlocalfields} 
using entirely different methods.  This is the $n=1$ case of Lemma
\ref{lem:33}.  
The results in this section generalize to any algebraically closed field of characteristic zero.
\end{rmk}

\subsection{The real place}
For the real numbers $\RR$,  
$\M_{\star}=\FF_2[\tau,\rho]$ as algebras, 
where $\vert\tau\vert=1-\alpha$, $\vert\rho\vert=-\alpha$.
In order to determine the $\Ext$-groups over $\mathcal{E}(n)$, 
we can run the $\rho$-Bockstein spectral sequence for $\rho\colon\Sigma^\alpha\M\rightarrow\M$. 
It is an example of the filtration spectral sequence in \cite[Theorem A 1.3.9]{Ravenelbook}.
This work has been carried out by Hill in \cite[Theorem 3.1]{Hill} who also carefully spells out the properties of the $\rho$-BSS.
By comparison with $\CC$, 
the $E_{1}$-term of the $\rho$-BSS takes the form
\[
\FF_2[\tau,\rho,v_{0},\ldots,v_{n}].
\]

\begin{prop}
\label{prop:ExtEnR}
Over $\RR$, the differentials
\[
  d_{2^{i+1}-1}\tau^{2^i} = \rho^{2^{i+1}-1}v_i,~0\le i\le n,
\]
determine the $\rho$-Bockstein spectral sequence computing $\Ext_{\mathcal{E}(n)}$.

As an algebra,
\[
\Ext^{\star,\ast}_{\mathcal{E}(n)}(\M_{\star},\M_{\star})
= 
\FF_2[\rho,\tau^{2^{n+1}},v_i(j) \mid 0\le i \le n, 0\le j]/(\rho^{2^{i+1}-1}v_i(j))
\]
subject to the additional relations
\[
  v_i(j)v_k(\ell) = v_i(j+2^{k-i}\ell)v_k(0)
\]
when $i\le k$ and
\[
  v_i(j) = \tau^{2^{n+1}}v_i(j-2^{n-i})
\]
when $j\ge 2^{n-i}$.  
Here $v_i(j)$ is represented on $E_{1}$ by $\tau^{2^{i+1}j}v_{i}$, 
and has degree
\[
(2^i(2j+1)-1 - (2^i(2j-1)+1)\alpha,1)
\]
\end{prop}

In Section \ref{sec:rational} it will be useful to have a thorough understanding of the combinatorics of this spectral sequence.  
It is difficult to visualize the computation when $n>1$ using standard conventions because the pictures become far too dense.  
We introduce a graphical calculus below that eliminates this
difficulty.  We analyze the case of $\mathcal{E}(3)$ in detail.  

\begin{rmk}\label{rmk:graph1}
Figure \ref{fig:R} is a graphical presentation of the spectral
sequence in Proposition \ref{prop:ExtEnR} when $n=3$.  We have drawn
the pictures of this quad-graded spectral sequence in only two
dimensions.  Recall that each element of the $E_r$ page of the
$\rho$-BSS has a homological grading $(s,m+n\alpha)$ and also its
$\rho$-power filtration.  We draw such an element in degree
$m-s+n\alpha$ where $m-s$ is plotted on the horizontal axis while
$n\alpha$ goes on the vertical axis.  As such, these pictures
represent the ``total Adams degree'' of the elements in question.  (If
they survive the Adams spectral sequence, this is the degree to which
they abut.)  The
authors prefer to think of the degree $s$ as secretly recorded on a
third axis coming out of the page, while $\rho$-filtration should be
kept track of privately as an extra decoration on each element.  In
this grading, differentials on the $E_r$-page of the $\rho$-BSS point
one to the left (with no vertical component) and out of the page one
unit as well; they increase the decoration by $\rho$-filtration by
$r$.

Note that this is not the ``standard'' Adams grading, which might draw
$m-s+n$ on the horizontal axis and $s$ on the vertical axis, with
weight $n\alpha$ and $\rho$-filtration suppressed.  We find our
pictures more convenient and useful for two reasons.  First, weight
information is often useful in limiting which differentials are
possibly nonzero in $\rho$-Bockstein and motivic Adams spectral sequences.
Second, these pictures do a better job of capturing the sort of
connectivity that the motivic spectra we study enjoy.  In ``standard''
Adams grading, the copy of mod 2 Milnor $K$-theory in $\M_\star$ takes
up the entire negative horizontal axis.  Since we like to think of
$k^M_*$ as ``dimension 0'' information in motivic homotopy, it feels
better to place it along the vertical axis.

The reader can now 
interpret the figure via the following key and
comments:

\begin{center}
\begin{tabular}{l|l}
$\circ$ & $\FF_2[v_0,v_1,v_2,v_3]$\\ \hline
$\bullet$ & $\FF_2[v_0,v_1,v_2,v_3]/v_0$\\ \hline
$\circledcirc$ & $v_0\FF_2[v_0,v_1,v_2,v_3]$\\ \hline
$\bullet^\smallsetminus$ & $\FF_2[v_0,v_1,v_2,v_3]/v_0,v_1$\\ \hline
$\bullet^{\smallsetminus{\kern-3pt}\smallsetminus}$ &
$\FF_2[v_0,v_1,v_2,v_3]/v_0,v_1,v_2$\\ \hline
$\bullet^{\smallsetminus{\kern-3pt}\smallsetminus{\kern-3pt}\smallsetminus}$
& $\FF_2[v_0,v_1,v_2,v_3]/v_0,v_1,v_2,v_3$\\ \hline
$\mid$ & $\rho$-multiplication
\end{tabular}
\end{center}

The authors find it convenient to think of each backslash as ``killing
off'' (or ``blocking'') a $v_i$-multiplication.

Note, though, that in Figure \ref{fig:R} we do not draw the
$v_i$-multiplications in the diagram unless a
$v_i$-multiple is the target of a differential.  In general, a copy of
$\FF_2[v_0,v_1,v_2,v_3]$ lies in a plane perpendicular to the
horizontal and vertical axes which intersects our pictures in a line
of slope $1$. (Within this perpendicular plane, $v_0$-multiplication is
vertical and $v_i$-multipliction has slope $1/(2^i-1)$.)  If a
$v_i$-multiple is the target of a differential, we
draw the appropriate line segment of slope $1$ and draw our
differentials hitting these classes; otherwise, $v_i$-multiplication
is only encoded by our system of circles, dots, and dots with
slashes.  We do this so that the pictures do not become unmanageably
cluttered.

In certain places there are
$v_i$-multiples, $i>0$, which are divisible by $v_0$.  This occurs on
the classes $v_0(j)$ when $j$ is even, and there they are represented
by dashed lines of slope 1 (possibly curved to avoid overlap).  For
example, $v_1v_0(4) = v_1(2)v_0$ so there is a dashed line joining
$v_0(4)$ and $v_1(2)$.  Similarly, there is a dashed line joining
$v_0(4)$ and $v_2(1)$ because $v_2v_0(4) = v_2(1)v_0$.

All labels in these pictures refer to the lowest cohomological degree
element in that Adams bi-grading.  So while the ``target" of the
differential on $\tau$ on the $E_1$-page is labeled $\rho$, the
differential in fact hits $\rho v_0$.  The necessary number of $v_0$'s
can be deduced from the cohomological degree of the source and the
page number of the spectral sequence.

To extend these pictures to larger $n$, the reader simply needs to
reinterpret $\circ$ as $\FF_2[v_0,\ldots,v_n]$ and $\bullet$ with $k$
slashes as $\FF_2[v_0,\ldots,v_n]/v_0,\ldots,v_k$.  If $n>3$, then
$\tau^{16}$ will support a differential and there will be a yet more
elaborate pattern on $\tau^{16}$ in the $E_\infty$ page.  Similar
statements hold for $\tau^{32}, \tau^{64}$, etc.
\end{rmk}

\begin{center}
\begin{figure}
\begin{tabular}{ccc}
\includegraphics[width=2.2in]{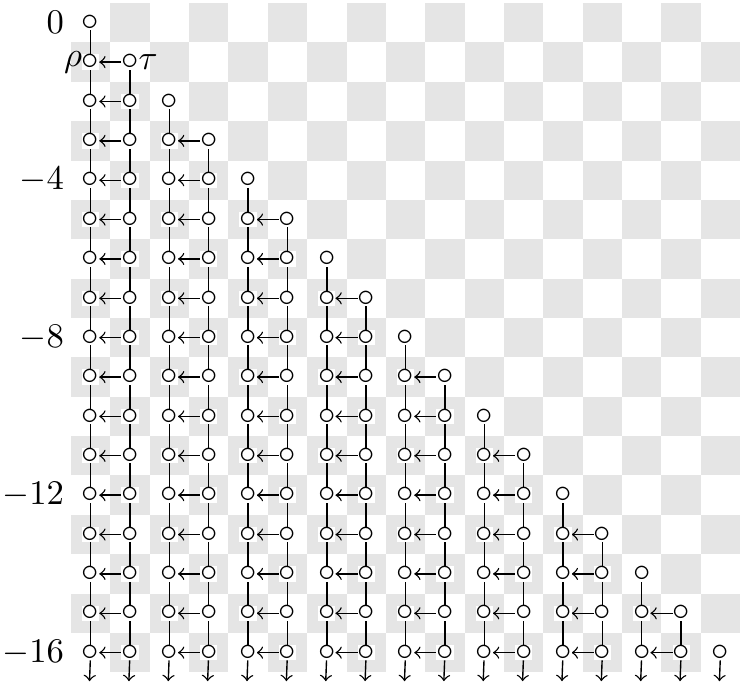} & ~~
&\includegraphics[width=2.2in]{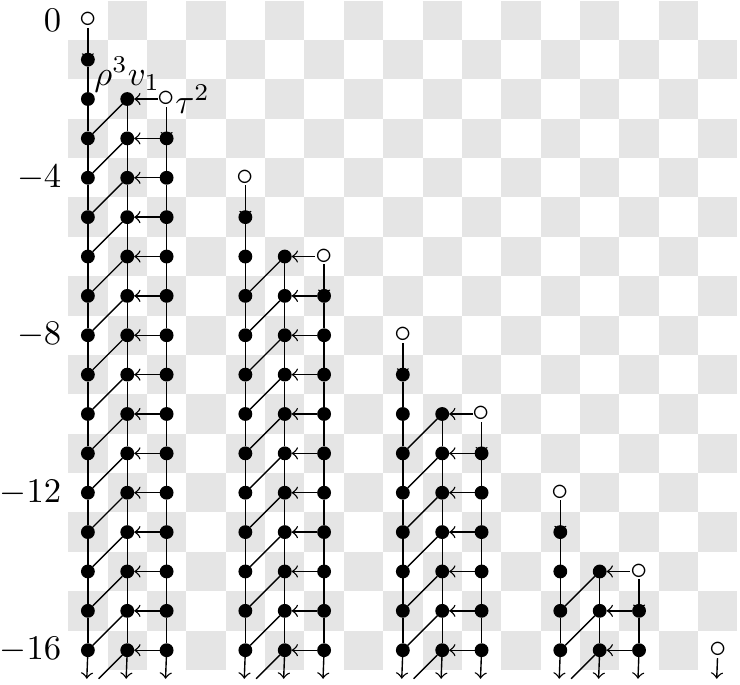}\\
$E_1$ && $E_3$
\end{tabular}

\begin{tabular}{ccc}
\includegraphics[width=2.2in]{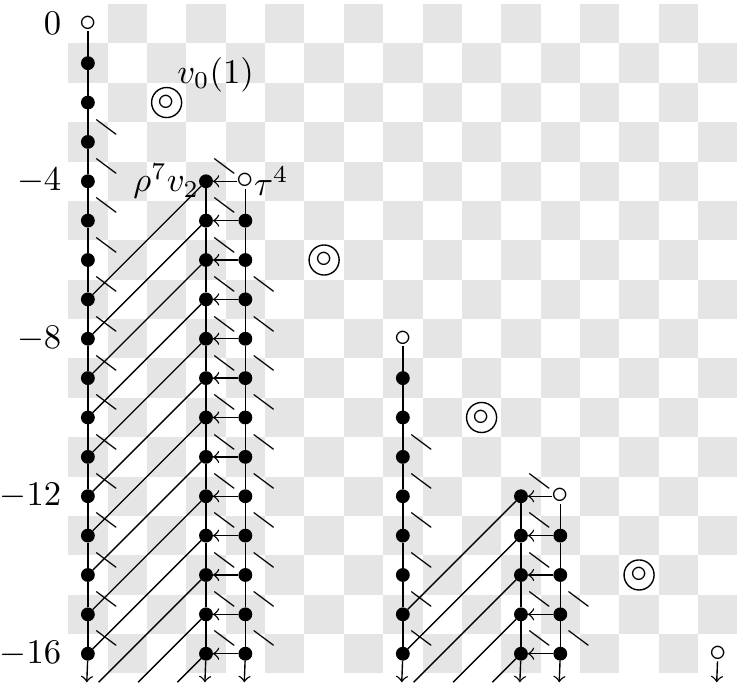} & ~~
& \includegraphics[width=2.2in]{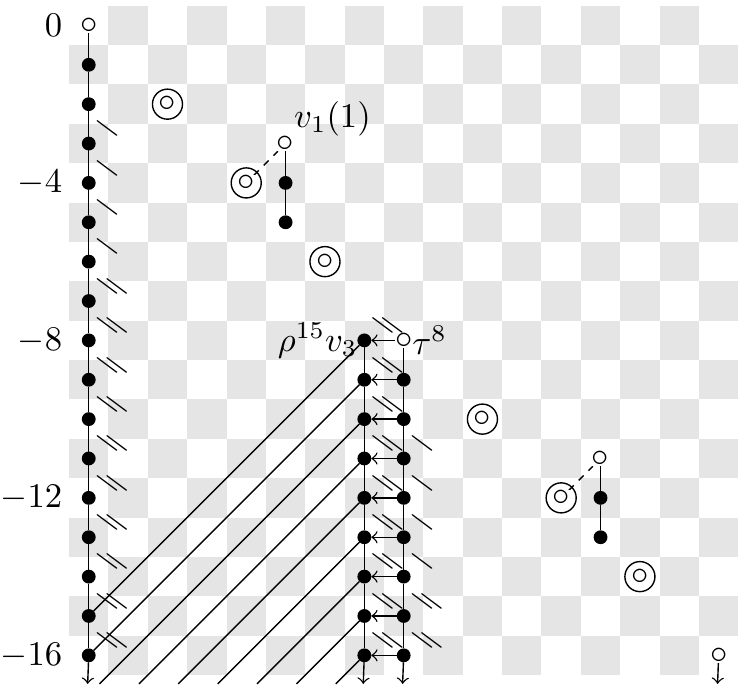}\\
$E_{7}$ && $E_{15}$
\end{tabular}

\includegraphics[width=2.2in]{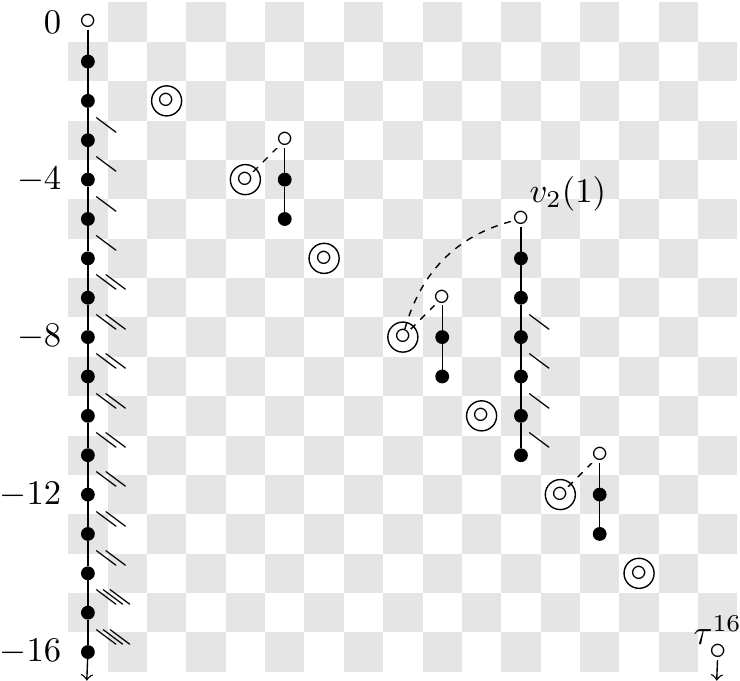}\\
$E_\infty$
\caption{The $\rho$-BSS for $\Ext_{\mathcal{E}(3)}$ over
  $\RR$.}\label{fig:R}
\end{figure}
\end{center}

Recall from \cite[Theorem 5.3, Corollary 5.7]{Hill} the collapse of
the MASS for $\BPn{n}$ over $\RR$.
Combined with the Real truncated Brown-Peterson spectra computations due to Hu in \cite{Hu}, 
appropriately amended,
we arrive at the following result.
\begin{thm}\label{thm:bpnR}
The motivic Adams spectral sequence for $\BPn{n}$ over $\RR$ collapses at $E_2$ and the homotopy 
$\BPn{n}_{\star}$ is given additively by 
\[
\ZZ_2[\rho,\tau^{2^{n+1}},v_i(j) \mid 0\le i \le n, 0\le j]
\]
subject to the relations $v_0(0)=2$, $\rho^{2^{i+1}-1}v_i(j)$, 
$v_i(j)v_k(\ell) = v_i(j+2^{k-i}\ell)v_k(0)$ when $i\le k$, and $v_i(j) = \tau^{2^{n+1}}v_i(j-2^{n-i})$ when $j\ge 2^{n-i}$.  
The degree of $v_i(j)$ is $(2^{i}-1)(1+\alpha)+2^{i+1}j(1-\alpha)$.  (If $n=\infty$ one should read the expression as lacking a $\tau$-power generator and having $v_i(j)$ generators for $0\le i<\infty,~0\le j$.)
\end{thm}

\begin{rmk}
While his methods are quite different, it should also be noted that Nobuaki Yagita produced similar computations for the homotopy of $\BP/2$ over $\RR$ via the Atiyah-Hirzebruch spectral sequence, cf.~\cite{Yagita}.
\end{rmk}

For the inclusion $i\colon\RR\subset\CC$ recall the identity map on $\BPn{n}_\CC$ induces a comparison map 
\[
\BPn{n}_\RR
\to 
i_{\ast}\BPn{n}_\CC.
\] 
(See Section \ref{sec:Hasse} if the above technology is unfamiliar.)  Applying the complex topological realization functor to $\BPn{n}_\RR$ yields the comparison with the 
topological truncated Brown-Peterson spectrum $\BPn{n}$. 

On homotopy groups, 
see also Proposition \ref{prop:ratlModelHtpy},
the first comparison map is determined by the following result.
\begin{lemma}
\label{RtoCcomparison}
The comparison map for $i\colon\RR\subset\CC$ induces a map between the motivic Adams spectral sequences for $\BPn{n}_\RR$
\[
\Ext_{\mathcal{E}(n)}^{s,m+n\alpha}(\M_{\ast},\M_{\ast})
\Longrightarrow
\pi_{m-s+n\alpha}\BPn{n}_\RR
\]
and for $\BPn{n}_\CC$
\[
\Ext_{\mathcal{E}(n)}^{s,m+n\alpha}(\M_{\ast},\M_{\ast})
\Longrightarrow
\pi_{m-s+n\alpha}\BPn{n}_\CC.
\]
It sends $\rho$ to $0$, $\tau^{2^{n+1}}$ to $\tau^{2^{n+1}}$, and $v_i(j)$ to $\tau^{2^{i+1}j}v_{i}$ for $0\leq i\leq n$.
\end{lemma}
\begin{rmk}
Comparing $v_{0}$-towers for the spectral sequences in Lemma \ref{RtoCcomparison} implies the map 
$K_{m}(\RR)\to K_{m}(\CC)$ is an isomorphism when $m\equiv 0\bmod 8$,
the multiplication by $2$ map on $\ZZ_{2}$ when $m\equiv 4\bmod 8$ and trivial otherwise.
\end{rmk}

\begin{lemma}
The complex topological realization functor induces a map between the motivic Adams spectral sequences for $\BPn{n}_\RR$
\[
\Ext_{\mathcal{E}(n)}^{s,m+n\alpha}(\M_{\ast},\M_{\ast})
\Longrightarrow
\pi_{m-s+n\alpha}\BPn{n}_\RR
\]
and the topological $\BPn{n}$ spectrum
\[
\Ext_{\mathcal{E}(n)^{\top}}^{s,t}(\HHH\ZZ/2_{\ast},\HHH\ZZ/2_{\ast})
\Longrightarrow
\pi_{t-s}\BPn{n}.
\]
It sends $\rho$ to $0$, $\tau^{2^{n+1}}$ to $1$, and $v_i(j)$ to $v_{i}$ for $0\leq i\leq n$.
\end{lemma}
\begin{rmk}
The real topological realization functor takes $\kgl_\RR$ to a trivial spectrum.
\end{rmk}

We can identify the weight zero subalgebra of $\pi_{\star}\BPn{1}_\RR$ with the coefficient ring of $2$-completed connective real topological $K$-theory.

\begin{lemma}\label{lemma:ko}
The subalgebra $\pi_{\ast+0\alpha}\BPn{1}_\RR$ is isomorphic to $\pi_{\ast}\ko$.
\end{lemma}
\begin{proof}
Recall the ring isomorphism
\[
\pi_{\ast}\ko
\cong
\ZZ_{2}[\eta,\alpha,\beta]/(2\eta,\eta^{3},\eta\alpha,\alpha^{2}-4\beta),
\]
where $\vert\eta\vert=1$, $\vert\alpha\vert=4$ and  $\vert\beta\vert=8$, 
cf.~\cite[Theorem 3.1.26]{Ravenelbook}.
We have
\[
\pi_{\star}\BPn{1}_\RR
= 
\ZZ_2[\rho,\tau^4,v_0(1),v_1]/(2\rho,\rho^3 v_1,v_0(1)^2-4\tau^4).
\]
The assertion follows by mapping $\pi_{\ast}\ko$ into $\pi_{\star}\BPn{1}_\RR$ by sending 
$\eta$ to $\rho v_{1}$, $\alpha$ to $v_0(1) v_{1}^{2}$ and $\beta$ to $\tau^4v_{1}^{4}$.
\end{proof}
\begin{rmk}
The isomorphism $K_{\ast}(\RR)\cong\pi_{\ast}\ko$ was shown by Suslin in \cite{Suslinlocalfields} 
using entirely different methods.
The results in this section generalizes to all real closed fields, 
e.g., the field $\overline{\QQ}\cap\RR$ of real algebraic numbers.
\end{rmk}
\begin{rmk}
Following the reasoning after the proof of Lemma 5.4 in \cite{Hill},
we can explain Lemma \ref{lemma:ko} by considering the realification
functor $t$ from $\PP^1$-spectra over $\RR$ to $\ZZ/2$-equivariant
spectra.  It should be the case that $t(\BPn{n}_\RR) = \BP\RR\langle
n\rangle$, where the spectrum $\BP\RR\langle n\rangle$ is the Real
truncated Brown-Peterson spectrum of Po Hu
\cite{Hu}.  This then induces a map
\[
  \pi_{*+0\alpha}\BPn{n}_\RR \to \pi_{*+0\sigma}\BP\RR\langle n\rangle
\]
where $\sigma$ is the sign representation of $\ZZ/2$ and we are
working with $RO(\ZZ/2)$-graded homotopy groups.  Now the target of
this map is easily identified with $\pi_* \BP\RR\langle
n\rangle^{\ZZ/2}$ and Hu shows in \cite{Hu} that
$\BP\RR\langle 1\rangle^{\ZZ/2} \simeq \ko$.  Hence the above map induces the
isomorphism of Lemma \ref{lemma:ko}
\end{rmk}

\subsection{Non-Archimedean places}
Let $p$ be an odd prime number.  In \cite{Ormsby}, the first author
determines the behavior of the $\rho$-Bockstein and motivic Adams
spectral sequences for $\BPn{n}$ over $\QQ_p$.  The differentials
observed are quite similar to those for $\BPn{n}$ over $\RR$, but the
$\rho$-BSS always collapses at $E_2$ and there are (infinitely many)
nontrivial differentials in the MASS.

Recall that
\[
  k^M_*(\QQ_p) =
  \begin{cases}
  \FF_2[u,p]/(u^2,p^2) &\text{if }p\equiv 1\pod{4},\\
  \FF_2[u,p]/(u^2,p(u-p)) &\text{if }p\equiv 3\pod{4},
  \end{cases}
\]
where $u$ is a nonsquare in the Teichm\"uller lift $\FF_p^\times\subset \QQ_p^\times$, 
and $p$ is the class of the uniformizer $p$.  
If $p\equiv 3\pod{4}$ we choose $u$ to be the class of $\rho$, the
class of $-1$, while $\rho=0$ when $p\equiv 1\pod{4}$.  
Recall that $\M_\star = k^M_*[\tau]$.

The following is the main result of \cite[\S 4]{Ormsby}.

\begin{thm}
If $p\equiv 1\pod{4}$ then the $\rho$-BSS for $\Ext_{\mathcal{E}(n)}$ over $\QQ_p$ collapses and
\[
  \Ext_{\mathcal{E}(n)}(\M_\star,\M_\star) = \M_\star[v_0,\ldots,v_n].
\]

If $p\equiv 3\pod{4}$ then the $\rho$-BSS for $\Ext_{\mathcal{E}(n)}$ is determined by the differential
\[
  d_1\tau = \rho v_0
\]
and
\[
\Ext_{\mathcal{E}(n)}(\M_\star,\M_\star) 
=
k^M_*[\tau^2,v_0,\ldots,v_n]/\rho v_0 \oplus \rho\tau k^M_*[\tau^2,v_0,\ldots,v_n].
\]
\end{thm}

Let $\varepsilon(p) = \nu_2(p-1)$ and $\lambda(p) = \nu_2(p^2-1)$ where $\nu_2$ is the 2-adic valuation.  
Then \cite[\S 5]{Ormsby} shows that the MASS for $\BPn{n}$ over $\QQ_p$ takes the following form.

\begin{thm}\label{thm:BPnQp}
If $p\equiv 1 \pod{4}$, 
then the MASS for $\BPn{n}$ over $\QQ_p$ is determined by differentials
\[
d_{\varepsilon(p)+i}\tau^{2^i} 
= 
u\tau^{2^i-1}v_0^{\varepsilon(p)+i}.
\]

If $p\equiv 3\pod{4}$, then the MASS for $\BPn{n}$ over $\QQ_p$ is
determined by
\[
d_{\lambda(p)-1+i}\tau^{2^i} 
= 
\rho\tau^{2^i-1}v_0^{\lambda(p)-1+i}.
\]
\end{thm}

For a description of the $E_\infty$ term see \cite[Theorem 5.7]{Ormsby}.

Set $k = \varepsilon(p)$ if $p\equiv 1\pod{4}$ and set $k =\lambda(p)$ if $p\equiv 3 \pod{4}$.  
The MASS for $\BPn{n}$ over $\QQ_p$ is depicted graphically in Figure
\ref{fig:Qp} using the same conventions as Figure \ref{fig:R} (see
Remark \ref{rmk:graph1}) with the
addition that $\circ/v_0^r = \FF_2[v_0,\ldots,v_n]/v_0^r$ and that
Adams spectral sequences no longer have a $\rho$-filtration grading.  
(That said, the $E_k$ page is actually the $\rho$-BSS if $k=1$.)

\begin{center}
\begin{figure}
\begin{tabular}{cc}
\includegraphics[width=1.8in]{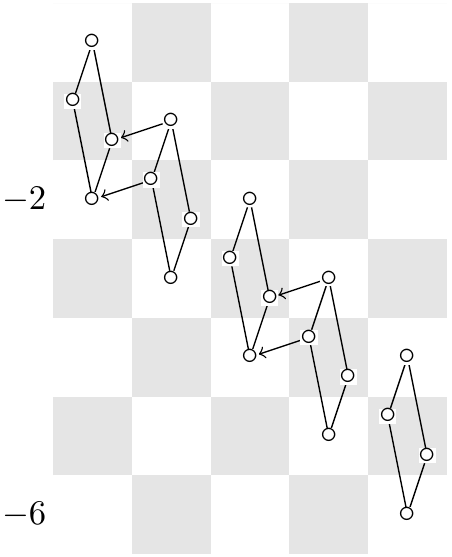}
&
\includegraphics[width=1.8in]{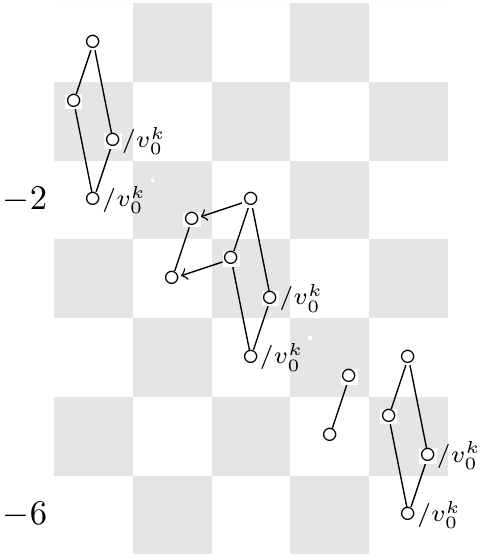}

\\

$E_{k}$	&$E_{k+1}$	\\

\includegraphics[width=1.8in]{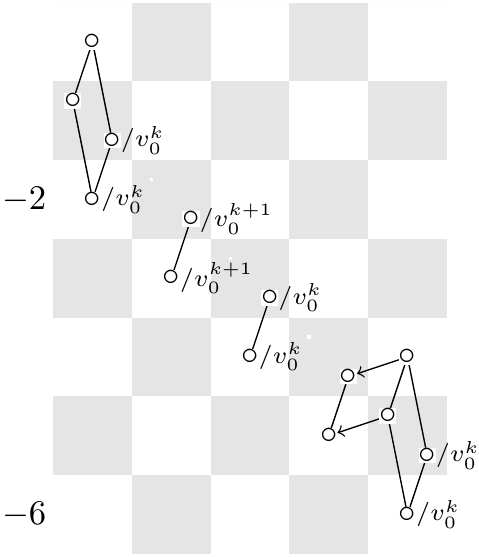}
&
\includegraphics[width=1.8in]{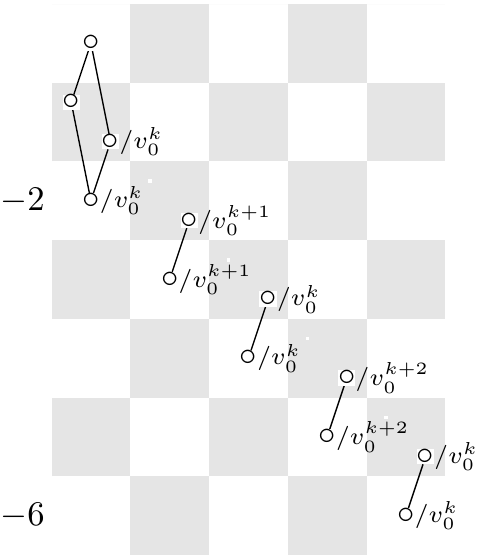} \\

$E_{k+2}$	&$E_{k+3}$\\
\end{tabular}
\caption{The MASS for $\BPn{n}$ over $\QQ_p$, $p>2$.}\label{fig:Qp}
\end{figure}
\end{center}

We now turn our attention to the field $\QQ_2$.  The first mod 2
Milnor $K$-theory of $\QQ_2$ is generated by the classes of $-1$, $2$,
and $5$, which we denote $\rho$, $x$, and $y$, respectively.  Then
\[
  k^M_*(\QQ_2) = \FF_2[\rho,x,y]/(\rho^3,x^2,y^2,\rho^2+xy,\rho x,\rho
  y).
\]

\begin{thm}
The $\rho$-BSS for $\Ext_{\mathcal{E}(n)}$ over $\QQ_2$ is determined
by the differential
\[
  d_1\tau = \rho v_0.
\]
\end{thm}
\begin{proof}
We compute $d_1\tau$ as $\eta_L(\tau) - \eta_R(\tau) = \rho v_0$.
Further differentials follow the pattern of Proposition
\ref{prop:ExtEnR}, but $\rho^3 = 0$ in $k^M_*(\QQ_2)$ so the spectral
sequence collapses at $E_2$.
\end{proof}

\begin{thm}
The MASS for $\BPn{n}$ over $\QQ_2$ is determined by the differentials
\[
  d_{2+i}\tau^{2^i} = x\tau^{2^i-1}v_0^{2+i}
\]
for $1\le i$.
\end{thm}
\begin{proof}
We first treat the case $n=0$ in which we utilize a comparison with \'{e}tale cohomology.  
Applying the universal coefficient theorem and results in \cite[Chapter VII]{NSW} yields
\[
(\M\ZZ_2)_{m+n\alpha} 
=
\begin{cases}
  \ZZ_2 &\text{if }m=n=0,\\
  \ZZ_2^{\oplus 2} \oplus \ZZ/2 & \text{if }m=0,n=-1,\\
  \ZZ/2 &\text{if }m=0,n=-2,\\
  \ZZ_2\oplus \ZZ/2 &\text{if }n=-(m+1)<-1,~m\text{ even},\\
  \ZZ_2\oplus \ZZ/2^{2+\nu_2(m+1)} &\text{if }n=-(m+1)<-1,~m\text{
    odd},\\
  \ZZ/2 &\text{if }n=-(m+2)<-2,~m\text{ even},\\
  \ZZ/2^{2+\nu_2(m+1)} &\text{if }n=-(m+2)<-2,~m\text{ odd},\\
  0 &\text{otherwise}.
\end{cases}  
\]
In order to have the correct 2-torsion in degrees of the form
$m-(m+1)\alpha$ we must have $d_{2+i}\tau^{2^i} = z\tau^{2^i-1}v_0^{2+i}$ 
for $z$ a nonzero linear combination of $x$ and $y$.  
In Lemma \ref{lem:x} we show that $z = x$.  
(The proof is deferred because it relies on the Hasse map defined in the next section.)

For $n>0$ consider the linearization map $\BPn{n}\to \BPn{0}$ and the induced map of spectral sequences.  
By tri-degree considerations we can identify the spectral sequence for $\BPn{n}$ as $E_*(\BPn{0})[v_1,\ldots,v_n]$ 
where the $v_{>0}$ are permanent cycles.  
This concludes the proof.
\end{proof}

The behavior of the spectral sequence is pictured in Figure
\ref{fig:Q2} using the same conventions as Figure \ref{fig:Qp}.  In the
dimensional range pictured $E_5 = E_\infty$.

\begin{center}
\begin{figure}
\begin{tabular}{cc}
\includegraphics[width=1.8in]{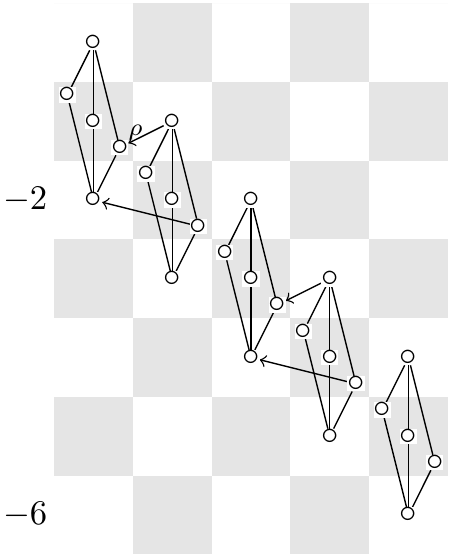}
&
\includegraphics[width=1.8in]{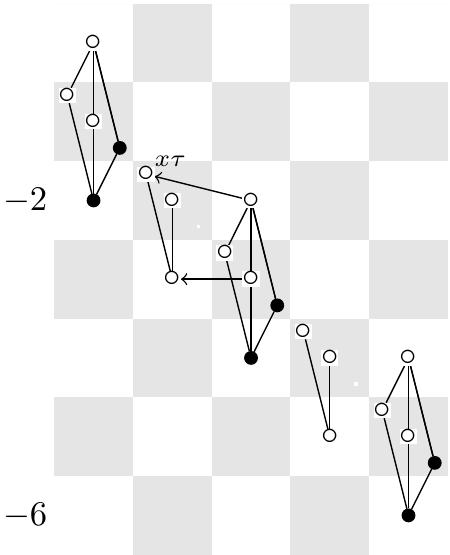}

\\

$E_1$	&$E_3$\\

\vspace{.6cm}\\

\includegraphics[width=1.8in]{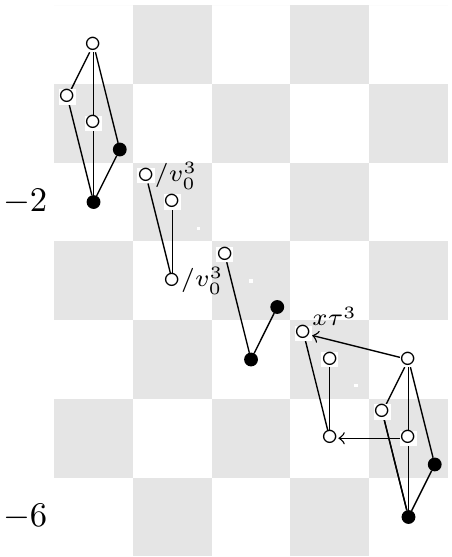}
&
\includegraphics[width=1.8in]{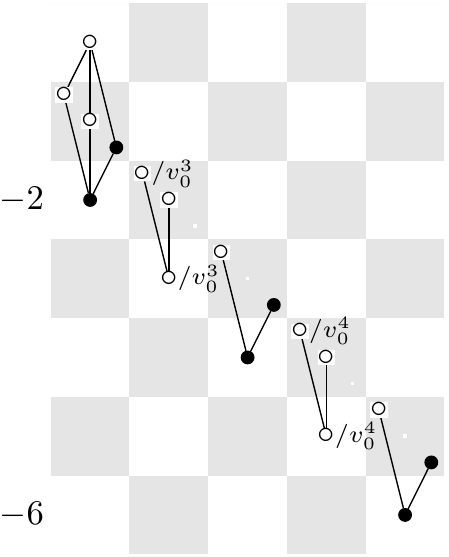}
\\

$E_4$	&$E_5$
\end{tabular}
\caption{The $\rho$-BSS and MASS for $\BPn{n}$ over
  $\QQ_2$}\label{fig:Q2}
\end{figure}
\end{center}

\begin{cor}
Over $\QQ_2$ we have $\BPn{n}_{\star}=(\M\ZZ_2)_{\star}[v_1,\cdots,v_n]$.
\end{cor}

\section{The Hasse map}\label{sec:Hasse}
Our goal now is to use the local data from the previous section to
produce global computations.  To this end, we need a method for
comparing the local and global variants of $\BPn{n}$.  In this section
we construct a Hasse map which serves this purpose.

For any scheme map $f:S\to T$ there is a base change functor
\[
  f^*:\Sp(T)\to \Sp(S)
\]
from $T$-spectra to $S$-spectra admitting a right adjoint
\[
  f_*:\Sp(S)\to \Sp(T).
\]

Let $\BPn{n}_S$ denote the truncated Brown-Peterson spectrum in $S$-spectra.  
Recall that $\BPn{n}_S$ is constructed from $\MGL_S$ by $p$-localizing, 
taking the iterated colimit of the Quillen idempotent (producing $\BP_S$) and then killing off $v_i$,
$i>n$.
Recall the isomorphism $\MGL_S \cong f^*\MGL_T$ for $f$ as above.

\begin{prop}\label{prop:BPnPullsBack}
The truncated Brown-Peterson spectra $\BPn{n}$ satisfy 
\[
\BPn{n}_S
\cong
f^*\BPn{n}_T.
\]
\end{prop}
\begin{proof}
Localization at $p$, inverting the Quillen idempotent, and killing $v_i$'s are all colimit constructions;
$f^*$ commutes with colimits because it is a left adjoint.
\end{proof}

For $v$ a real or $p$-adic place of $\QQ$ we denote by $i_v:\Spec \QQ_v \to \Spec\QQ$ the map of 
Zariski spectra induced by the field extension $\QQ_v/\QQ$.

Consider a family of motivic spectra (over different base fields)
$\EEE = \{\EEE_k = \EEE_{\Spec k} \in \SH(k)\mid
k=\QQ\text{ or }\QQ_v\}$.
In the following we assume that this family satisfies $\EEE_{\QQ_v}
\cong i_v^*\EEE_\QQ$ for all non-complex places $v$ of $\QQ$.

\begin{defn}\label{defn:ratlModel}
The \emph{rational model of $\EEE_{\QQ_v}$} is defined as 
\[
\EEE^\QQ_{\QQ_v} 
\defined 
i_v{}_*\EEE_{\QQ_v} 
\cong 
i_v{}_*i_v^*\EEE_\QQ.
\]
\end{defn}

The terminology is justified by the following proposition.

\begin{prop}\label{prop:ratlModelHtpy}
The bigraded homotopy groups of $\EEE^\QQ_{\QQ_v}$ (computed in $\SH(\QQ)$, the stable motivic
homotopy category over $\QQ$) are isomorphic to those of $\EEE_{\QQ_v}$ in $\SH(\QQ_v)$.
\end{prop}
\begin{proof}
By adjunction isomorphisms and the fact that $i_v^*\one_\QQ = \one_{\QQ_v}$ we have
\[\begin{aligned}
     \pi_\star^\QQ \EEE^\QQ_{\QQ_v} 
  = [\one_\QQ,\EEE^\QQ_{\QQ_v}]^\QQ_\star 
  = [\one_\QQ,i_v{}_*\EEE_{\QQ_v}]^\QQ_\star
  = [i_v^*\one_\QQ,\EEE_{\QQ_v}]^{\QQ_v}_\star 
  = [\one_{\QQ_v},\EEE_{\QQ_v}]^{\QQ_v}_\star.
\end{aligned}\]
This is by definition the homotopy groups $\pi_\star^{\QQ_v}\EEE_{\QQ_v}$ of $\EEE_{\QQ_v}$ in $\SH(\QQ_v)$.
\end{proof}

Since $\EEE^\QQ_{\QQ_v}\cong i_v{}_*i_v^*\EEE_\QQ$ the $(i_v^*,i_v{}_*)$
adjunction unit induces a map
\[
  \eta_v:\EEE_\QQ\to \EEE_{\QQ_v}^\QQ.
\]

\begin{defn}\label{defn:Hasse}
For $\EEE$ as above, the \emph{motivic Hasse map} is given by 
\[
H_{\EEE} 
\defined 
\prod \pi_\star\eta_v:\pi_\star\EEE_\QQ
\to 
\prod \pi_\star\EEE_{\QQ_v}^\QQ
\]
where the product runs over real and $p$-adic places $v$.
\end{defn}

\begin{defn}
The family of spectra $\EEE$ satisfies the \emph{motivic Hasse principle} 
if the motivic Hasse map $H_{\EEE}$ is monic.
\end{defn}

\begin{prop}\label{prop:HasseOnHtpy}
The Hasse map takes the form
\[
  H_{\EEE}:\pi_\star^\QQ \EEE_\QQ\to \prod \pi_\star^{\QQ_v} \EEE_{\QQ_v}.
\]
\end{prop}
\begin{proof}
This is a consequence of Proposition \ref{prop:ratlModelHtpy}.  
Note that the target cannot be pared down to a direct sum.  For instance, if we take $\EEE = \one$ we find that $\rho$ has nontrivial image in infinitely many of the groups $\pi_{-\alpha}^{\QQ_v} \one_{\QQ_v}$.
\end{proof}

The Hasse maps of interest in this paper are the ones for $\BPn{n}$, $0\le n\le \infty$.  
Let $E_*^{*,\star}(k)$ denote the MASS for $\BPn{n}_k$.  
Then the Hasse map induces a map of MASSs
\begin{equation}\label{eqn:MASSHasse}
  E_*^{*,\star}(\QQ) \to \prod_{v} E_*^{*,\star}(\QQ_v).
\end{equation}

\section{Computations over $\QQ$ and the motivic Hasse principle}\label{sec:rational}

This section proves the main theorem of this paper.

\begin{thm}\label{thm:Hasse}
The truncated Brown-Peterson spectra $\BPn{n}$ satisfy the motivic Hasse principle.
\end{thm}

\begin{cor}
There are no hidden multiplicative extensions in the motivic Adams
spectral sequence for $\BPn{n}$ over $\QQ$ and $v_0$-multiplication
represents multiplication by 2. 
\end{cor}
\begin{proof}
The computations with local fields in Section \ref{sec:comp} show the result holds for all completions of $\QQ$. 
Thus our claim follows from Theorem \ref{thm:Hasse}.  (Note that, in
general, $v_0 = 2+\rho\eta$, but $\eta = 0$ in $\pi_\alpha \BPn{n}$.)
\end{proof}

Our proof of Theorem \ref{thm:Hasse} follows from an analysis of
(\ref{eqn:MASSHasse}).  Carrying that analysis a few steps further we
also get a computation of $\pi_\star\BPn{n}$ over $\QQ$, $0\le n\le
\infty$, which is stated in Theorem \ref{thm:BPnQ} below.

To get these computations off the ground we need a detailed
understanding of $k^M_*(\QQ)$ and the Hasse map
\[
  k^M_*(\QQ) \to \prod_v k^M_*(\QQ_v).
\]
The following proposition consists of basic facts easily deduced from,
e.g., \cite[Example 1.8 and Appendix]{MilnorK}.

\begin{prop}\label{prop:kMQ}
The mod 2 Milnor K-theory of $\QQ$ has the following structure:
\[
\begin{aligned}
  k^M_0(\QQ) &= \ZZ/2,\\
  k^M_1(\QQ) &= \ZZ/2\{\rho\} \oplus \bigoplus_{p\ge 2} \ZZ/2\{[p]\},\\
  k^M_2(\QQ) &= \ZZ/2\{\rho^2\} \oplus \bigoplus_{p\ge 3}
  \ZZ/2\{a_p\},\\
  k^M_n(\QQ) &= \ZZ/2\{\rho^n\}\text{ if }n\ge 3.
\end{aligned}
\]
Multiplication follows the rule
\[
  [\ell]\cdot [q] = (\ell,q)_2\rho^2 + \sum_{p\ge 3}(\ell,q)_p a_p
\]
where $\ell$ and $q$ are primes or $-1$, $(~,~)_p$ is the Hilbert symbol if $p\ge 3$, and $(~,~)_2$ is the
2-adic Steinberg symbol.

The Hasse map takes pure symbols to their obvious images in $\prod
k^M_1(\QQ_v)$ and takes $a_p$ to the unique nonzero class in
$k^M_2(\QQ_p)$ and to $0$ in $k^M_2(\QQ_\ell)$, $\ell\ne p$.
\end{prop}

It will be convenient to understand the
$\rho$-module structure of $k^M_*(\QQ)$.

\begin{prop}\label{prop:rhoModule}
The $\rho$-module structure of $k^M_*(\QQ)$ is such that $\rho$ is not nilpotent and
\[
\begin{aligned}
  \rho\cdot [p] &= 0 \text{ if }p\equiv 1\pod{4},\\
  \rho \cdot [p] &= \rho^2 + a_p \text{ if }p\equiv 3\pod{4},\\
  \rho\cdot [2] &= 0.
\end{aligned}
\]
\end{prop}
\begin{proof}
The class $\rho$ is non-nilpotent because $\QQ$ has a real embedding.
(Note that Propositions \ref{prop:kMQ} and \ref{prop:rhoModule} omit a
few multiplicative relations in $k^M_*(\QQ)$ that are not important in
any of our calculations.)

The relations follow from computations of Hilbert and 2-adic Steinberg symbols.  Recall that for $a = 2^\alpha u$, $b = 2^\beta v$ for $u,v$ odd we have
\[
  (a,b)_2 = (-1)^{\frac{u-1}{2}\frac{v-1}{2}+\alpha\frac{v^2-1}{8}+\beta\frac{u^2-1}{8}}.
\]
Hence
\[
  (-1,p)_2 = (-1)^{-\frac{p-1}{2}} =
  \begin{cases}
    1&\text{if }p\equiv 1\pod{4},\\
    -1&\text{if }p\equiv 3\pod{4}
  \end{cases}
\]
for $p$ odd while $(-1,2)_2 = 1$.

To compute the $\ell$-adic Hilbert symbol for $\ell$ odd, write $a = \ell^\alpha u$, $b = \ell^\beta v$ for $p\nmid u,v$.  Then
\[
  (a,b)_\ell = (-1)^{\alpha\beta\frac{\ell-1}{2}} \left(\frac{u}{\ell}\right)^\beta \left(\frac{v}{\ell}\right)^\alpha
\]
where $\left(\frac{~}{~}\right)$ denotes the Legendre symbol.  Hence for $p$ odd
\[
  (-1,p)_p = \left(\frac{-1}{p}\right) =
  \begin{cases}
    1&\text{if }p\equiv 1\pod{4},\\
    -1&\text{if }p\equiv 3\pod{4}
  \end{cases}
\]
by the first supplement to quadratic reciprocity.  We also have
$(-1,2)_p = 1$.  This is enough to check the relations by Proposition \ref{prop:kMQ}.
\end{proof}

Following the usual pattern, we begin our motivic computations with the $\rho$-BSS computing
$\Ext_{\mathcal{E}(n)}$.

\begin{thm}\label{thm:rhoBSSQ}
The $\rho$-BSS for $\Ext_{\mathcal{E}(n)}$ over $\QQ$ is determined by
the differentials
\[
  d_{2^{i+1}-1}\tau^{2^i} = \rho^{2^{i+1}-1}v_i
\]
for $0\le i\le n$.  In addition to the obvious differentials that are
also present in the $\rho$-BSS over $\RR$, $d_1\tau =
\rho v_0$ induces differentials
\[
\begin{aligned}
  d_1 [p]\tau &= 0\text{ if }p\equiv 1\pod{4}\text{ or }p=2,\\
  d_1 [p]\tau &= (\rho^2+a_p) v_0\text{ if }p\equiv 3\pod{4}.
\end{aligned}
\]
\end{thm}

Most of this spectral sequence looks exactly like the $\rho$-BSS for
$\Ext_{\mathcal{E}(n)}$ over $\RR$, and the portion pertaining to
classes involving $p\equiv 3\pod{4}$ is depicted graphically in the
first part of Figure \ref{fig:A}.

\begin{proof}
This follows from the arguments of \cite[Theorem 3.2]{Hill} and Proposition \ref{prop:rhoModule}.
\end{proof}

\begin{cor}\label{cor:E2Hasse}
On $E_2$-terms, the Hasse map (\ref{eqn:MASSHasse}) of motivic Adams
spectral sequences for $\BPn{n}$ is injective.
\end{cor}
\begin{proof}
It is clear that the Hasse map is injective on $E_1$-terms of
$\rho$-Bockstein spectral sequences.  To show that injectivity is
preserved by the spectral sequences, we must show that every local boundary
which is in the image of the Hasse map is in fact the Hasse image of a
global boundary.  This is obvious by inspection of the spectral
sequences in question.
\end{proof}

Before moving on to the proof of Theorem \ref{thm:Hasse} we use
(\ref{eqn:MASSHasse}) and our computation of the $E_2$-term of the
MASS for $\BPn{n}$ over $\QQ$ to tie up a loose end from Section
\ref{sec:comp}.  Note that thus far none of the results in this
section have depended on the following lemma.

\begin{lemma}\label{lem:x}
In the MASS for $\BPn{n}$ over $\QQ_2$, the differentials take the form
\[
  d_{2+i}\tau^{2^i} = z\tau^{2^i-1}v_0^{2+i}
\]
with $z = x\in k^M_1(\QQ_2)$.
\end{lemma}
\begin{proof}
Assume for contradiction that $z = \epsilon x + y$ with $\epsilon = 0$ or $1$.  Recall that $x = [2]$ and $y = [5]$.  We claim that $\tau^{2^i}v_0$ survives to the $E_{2+i}$ page of the MASS for $\BPn{n}$ over $\QQ$.  Let $H^{\QQ_v}$ denote the projection of $H_{\BPn{n}}$ onto the $\QQ_v$ factor in the image of the Hasse map.  Then $H^{\QQ_5}d^\QQ_{2+i}\tau^{2^i}v_0 = d^{\QQ_5}_{2+i}\tau^{2^i}v_0 = (\epsilon [2]+[5])\tau^{2^i}v_0^{3+i}$.  From Theorem \ref{thm:BPnQp}, though, we know that this differential actually hits $[2]\tau^{2^i}v_0^{3+i}$, producing a contradiction.  (Note that $2$ generates the Teichm\"uller lift in $\QQ_5$.)

It remains to show that $\tau^{2^i}v_0$ survives to the $E_{2+i}$-page when working over $\QQ$.  Suppose not.  For tri-degree reasons, $\tau^{2^i}v_0$ would have to support a differential (rather than being the target), and it cannot support any lower differentials because they would be detected locally yet no such local differentials exist.
\end{proof}

We now aim to prove Theorem \ref{thm:Hasse} via analysis of
(\ref{eqn:MASSHasse}).

\begin{proof}[Proof of Theorem \ref{thm:Hasse}]
As in the proof of Corollary \ref{cor:E2Hasse}
we must show that on each page of the MASS every local boundary in the
image of $H_{\BPn{n}}$ is in fact the Hasse image of a global
boundary.  For induction, suppose that we have proven injectivity
on the $E_r$-page for some $r\ge 2$.  Suppose that $d_r(x) = y =
H(\tilde{y})$ in $\prod_v E_r^{*,\star}$.  We must verify that there is a global
differential $d_r(\tilde{x}) = \tilde{y}$.

We first go about constructing $\tilde{x}$ as an element of $E_2$ of the global MASS.  Note that the MASS over $\RR$ collapses and differentials over $\QQ_p$ take the form
\begin{equation}\label{eqn:a}
  d_r^{\QQ_p} \tau^{2^ij}\underline{v}^K = u'\tau^{2^ij-1}v_0^r\underline{v}^K
\end{equation}
or
\begin{equation}\label{eqn:b}
  d_r^{\QQ_p} [p]'\tau^{2^ij}\underline{v}^K = a_p'\tau^{2^ij-1}v_0^r\underline{v}^K.
\end{equation}
Here $\underline{v}^K$ is a monomial in the $v_i$, $u'$ generates the
Teichm\"uller lift (unless $p=2$ when $u' = x = [2]$), $[p]' = [p]$
(unless $p=2$ when $[2]' = y = [5]$), and $a_p' = a_p$ (unless $p=2$
when $a_p'=\rho^2=xy$).  Recall that $d_r(x) \in \prod_v E_r^{*,\star}$ has one coordinate for
each place of $\QQ$.  The coordinates of $d_r(x) = y$ are either all of the form (\ref{eqn:a}) or all of the form (\ref{eqn:b}).  In case (\ref{eqn:a}) we define $\tilde{x} = \tau^{2^ij}\underline{v}^K$, a well-defined element of $E_2$.  In case (\ref{eqn:b}), we know that $y$ has at most finitely many nonzero coordinates since $H:k^M_2(\QQ)\to \bigoplus_{p\ge 2} k^M_2(\QQ_p)$ is an isomorphism and $y = H(\tilde{y})$.  In this case we define $\tilde{x}$ to be the sum of the elements $[p]'\tau^{2^ij}\underline{v}^K$ for which the associated coordinate of $y$ is nonzero.

As long as $\tilde{x}$ survives to $E_r$ we are
guaranteed that $d_r(\tilde{x}) = \tilde{y}$ by the inductive
hypothesis.  A consideration of tri-degrees quickly verifies that $\tilde{x}$ is
not a boundary in any of the $E_{r'<r}$-pages of the MASS.  Hence we
only need to show that $d_{r'}\tilde{x} = 0$ for each $r'<r$.  By the form
of $E_2$, we know that $\tilde{x} = s \tau^t v_i(j)$ for some Milnor
symbol $s$.  Given the structure of the local MASS $E_2$, we may assume
$\tilde{x}$ is in fact of the form $\sum [p]\tau^{2^i}$, some $p\equiv 1\pod{4}$,
or $v_i(j)$ for $0\le i\le n$, $0\le j$.

In the first case, all
potential targets for $\sum [p]\tau^{2^i}$ are either
$k^M_2(\QQ)\cdot v_0^{r'}$-multiples of $\tau^{2^i-1}$ or are
$\rho^3\cdot v_{>0}$-divisible.  (This is easy to see by looking at
Figure \ref{fig:R} and adding in $k^M_*(\QQ_p)$ classes.)  The differentials in the latter class
would be witnessed by the MASS over $\RR$ which collapses, a
contradiction.  For the first class, if $r'<r$ then the Hasse map will
send the differential to a local $d_{r'}$-differential on $x$, a
contradiction.

In the second case, all potential targets for $v_i(j)$ are either
$k^M_1(\QQ)\cdot v_0^{r'}v_i(0)$-multiples of $\tau^{2^{i+1}-1}$ or
are $\rho^3\cdot v_{>0}$-divisible.  The same arguments go through,
proving the injectivity of (\ref{eqn:MASSHasse}).
\end{proof}

\begin{thm}\label{thm:BPnQ}
Let $\varepsilon(p) = \nu_2(p-1)$, the 2-adic valuation of $p-1$, and
let $\lambda(p) = \nu_2(p^2-1)$.  The MASS for $\BPn{n}$ over $\QQ$ is determined by
\[
\begin{aligned}
  d_{\lambda(p)+i-1} [p]\tau^{2^i} &= a_p
  \tau^{2^i-1}v_0^{\lambda(p)+i-1}\text{ if }p\equiv 3\pod{4},~i\ge 1,\\
  d_{\varepsilon(p)+i} [p]\tau^{2^i} &= a_p
  \tau^{2^i-1}v_0^{\varepsilon(p)+i}\text{ if }p\equiv 1\pod{4},~i\ge 0,\\
  d_{3+i+\nu_2(j)} v_i(j) &= [2] \tau^{2^{i+1}j-1}v_0^{3+i+\nu_2(j)}v_i(0)\text{ if
  }i\ge 0,j\ge 1.
\end{aligned}
\]
\end{thm}

Before proving the theorem we make several comments.

\begin{rmk}\label{rmk:graph2}
The behavior of the above spectral sequence (and that of Theorem
\ref{thm:rhoBSSQ}) is depicted in the Figures \ref{fig:A},
\ref{fig:B}, and \ref{fig:C} in the $n=3$ case.  They employ
the same graphical calculus employed in the pictures of the
$\rho$-Bockstein spectral sequence over $\RR$, and we strongly
recommend the reader review Remark \ref{rmk:graph1} before attempting
to digest these diagrams.  (Note that MASS pictures no longer have an
additional decoration by $\rho$-filtration, but the rest of the
grading is the same.)  We have split the
action of the spectral sequence into three digestible chunks:
differentials involving primes congruent to 3 mod 4 (Figure \ref{fig:A}), those for primes
congruent to 1 mod 4 (Figure \ref{fig:B}), and those present involving the prime 2 and the
real place (Figure \ref{fig:C}).

We have further compressed the diagrams by displaying all MASS
differentials simultaneously.  An arrow labeled by $r$ is a $d_r$
differential.

Finally, Figures \ref{fig:A} and \ref{fig:B} should really be
viewed as products of diagrams over primes congruent to 3 mod 4 and 1
mod 4, respectively, and $\lambda$ and $\varepsilon$ refer to
$\lambda(p) = \nu_2(p^2-1)$ and $\varepsilon(p) = \nu_2(p-1)$,
respectively.  The classes in degree $-\alpha$ represent $[p]$
and in degree $-2\alpha$ we have $a_p$.  The strings of elements
extending with slope $-1$ represent $\tau$-multiplication.  The
notation $\frac{\circ}{v_0^r}$ represents the algebra
$\FF_2[v_0,\ldots,v_n]/v_0^r$.
\end{rmk}

\begin{center}
\begin{figure}
\begin{tabular}{ccc}
\includegraphics[width=2.2in]{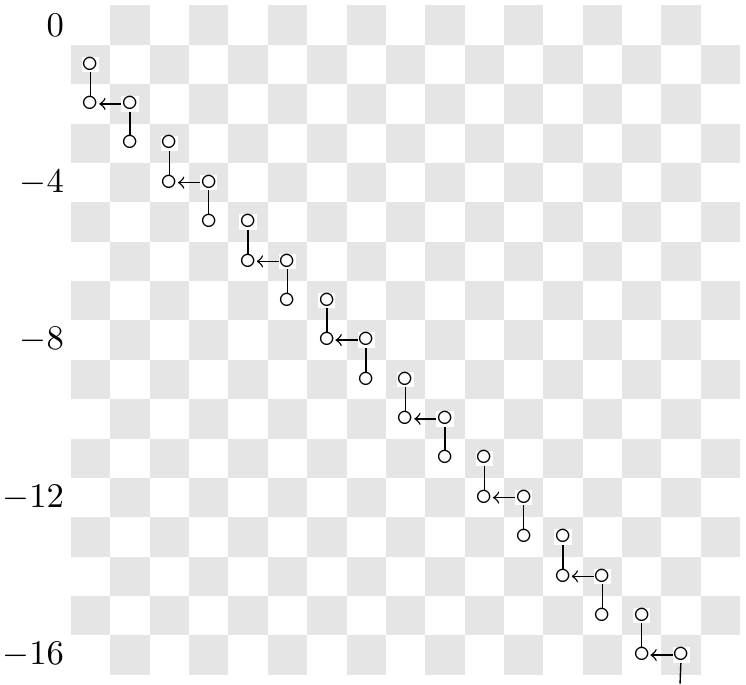} &~~&
\includegraphics[width=2.2in]{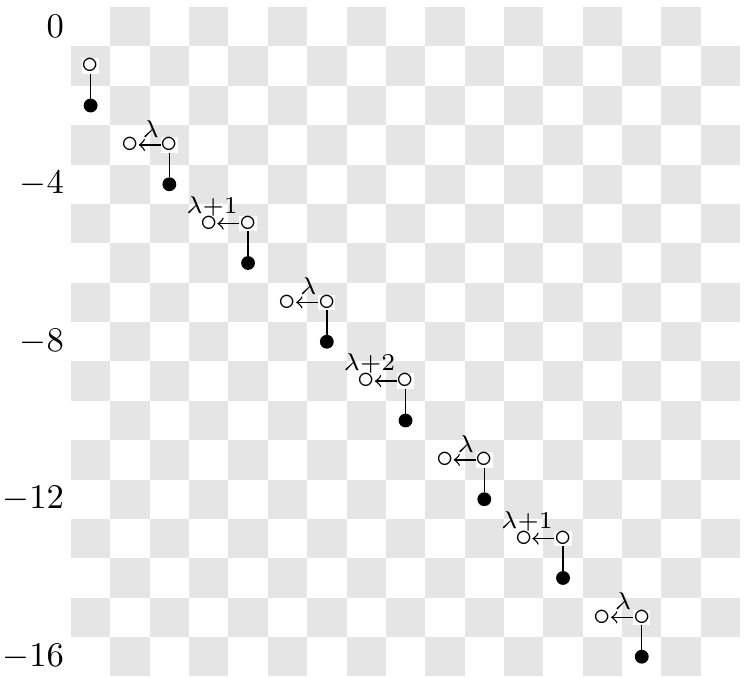}\\
$\rho$-BSS &~~& MASS
\end{tabular}

\includegraphics[width=2.2in]{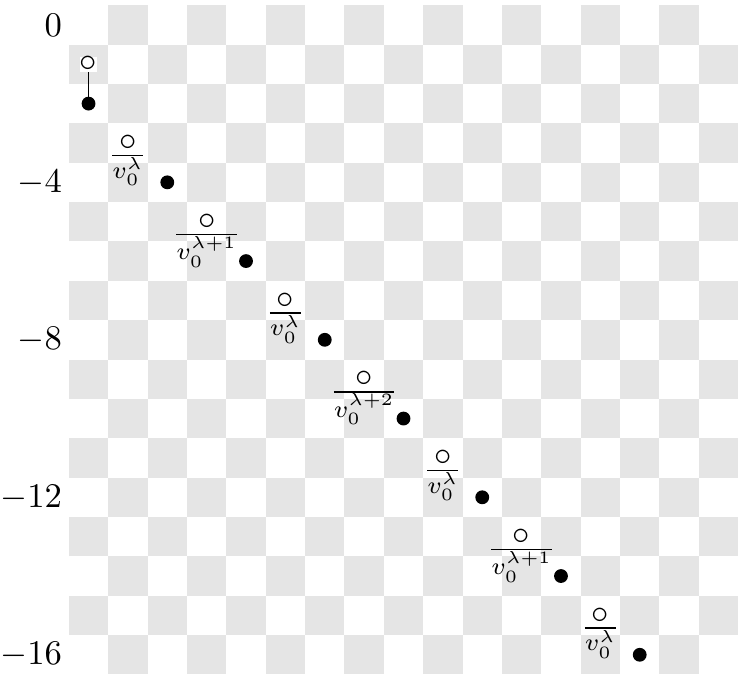}\\
$E_\infty$
\caption{Portions of the $\rho$-BSS and MASS for $\BPn{3}_\QQ$ involving
  $p\equiv 3\pod{4}$}\label{fig:A}
\end{figure}

\begin{figure}
\begin{tabular}{ccc}
\includegraphics[width=2.2in]{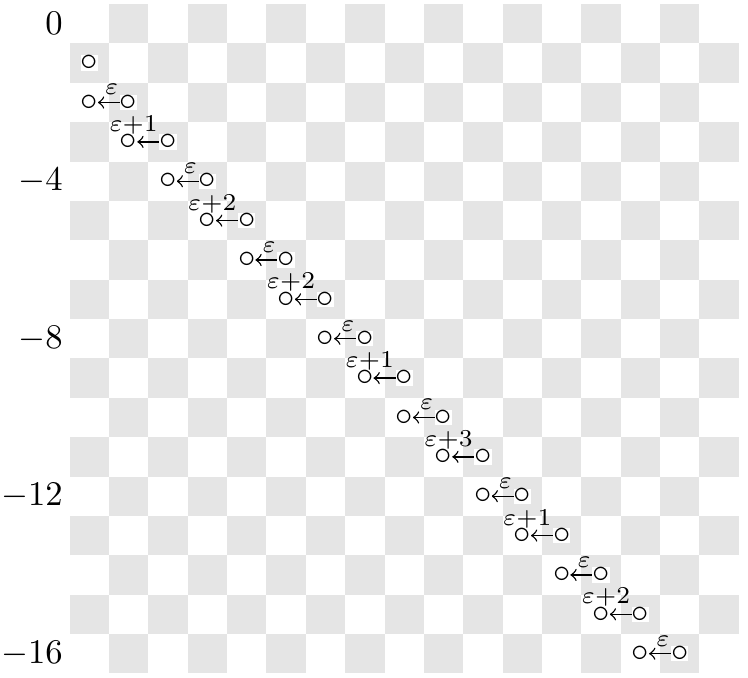} &~~&
\includegraphics[width=2.2in]{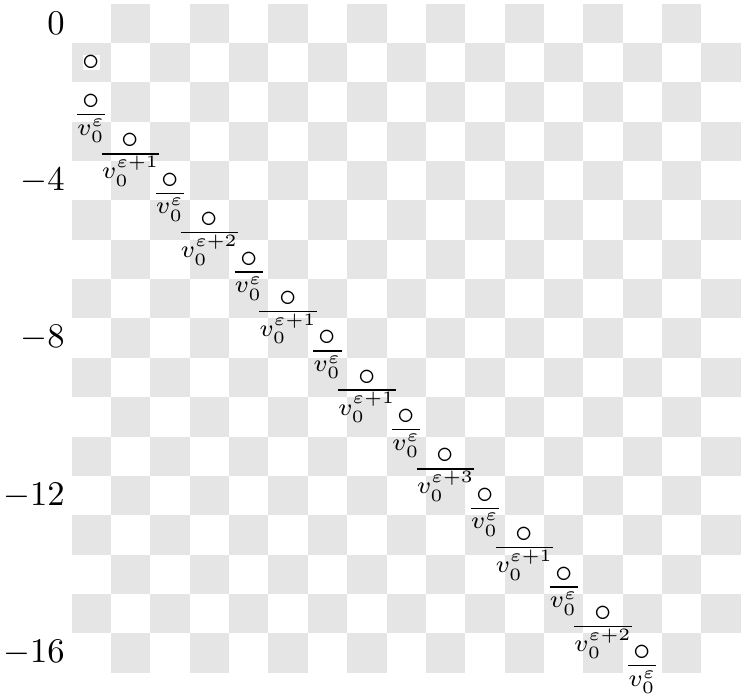}\\
MASS &~~& $E_\infty$
\end{tabular}
\caption{Portions of the MASS for $\BPn{3}_\QQ$ involving $p\equiv 1\pod{4}$.}\label{fig:B}
\end{figure}

\begin{figure}
\begin{tabular}{ccc}
\includegraphics[width=2.2in]{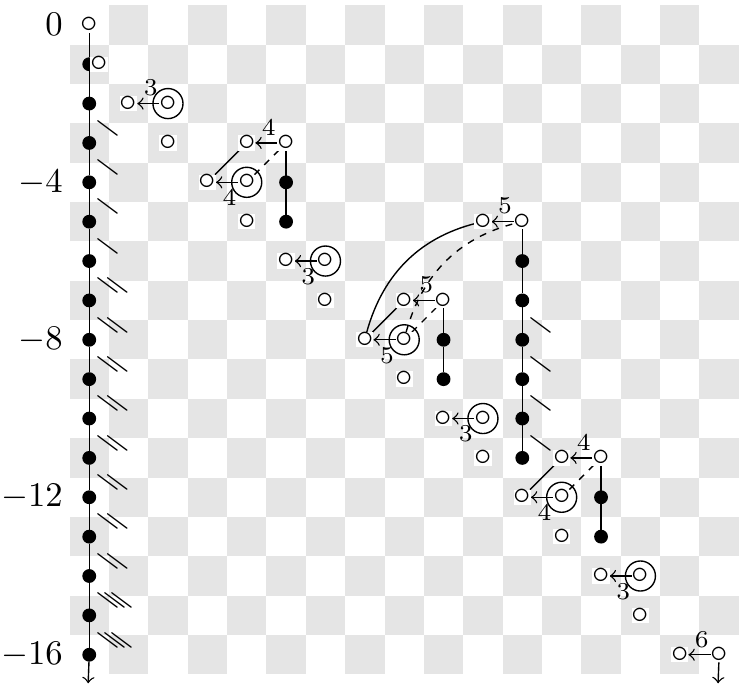} &~~&
\includegraphics[width=2.2in]{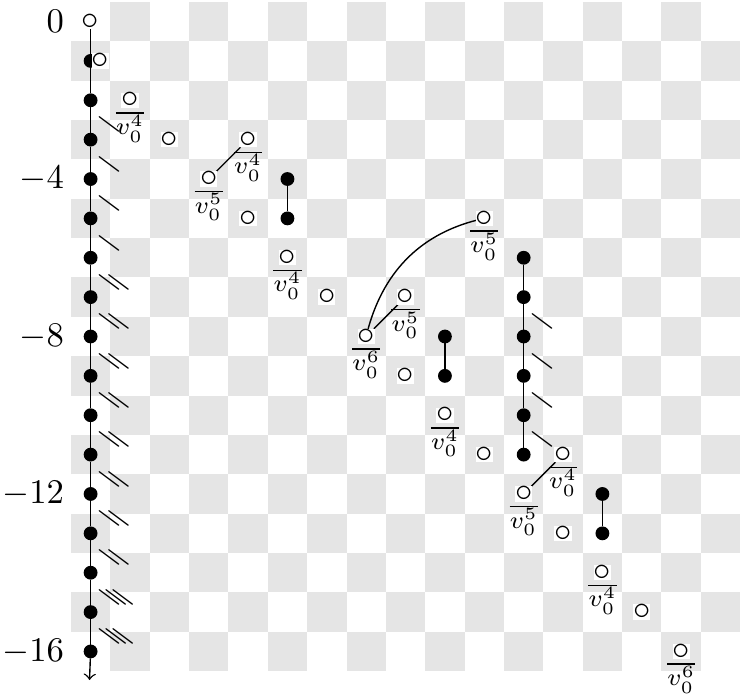}\\
MASS &~~&$E_\infty$
\end{tabular}
\caption{Portions of the MASS for $\BPn{3}_\QQ$ involving $2$-adic and
  real places.}\label{fig:C}
\end{figure}
\end{center}

\begin{rmk}
As the reader will note, the most interesting action in the
computation takes place when the real and 2-adic places intermingle.
Note that a $d_r$ differential on a $v_0(j)$ class produces
$v_0^{r+1}$-torsion in the target.  This happens because $v_0(j)$ is
in Adams filtration 1, not 0.

Consider, though, what happens to a class like $[2]\tau^3$ and its
$v_i(j)$-multiples.  (Note that $[2]\tau^3$ can
be located in degree $3-4\alpha$ in the MASS portion of Figure
\ref{fig:C}.)  We have a differential $d_4(v_0(2)) = [2]\tau^3v_0^5$,
making $[2]\tau^3$ a $v_0^5$-torsion class.  We also have $d_4(v_1(1))
= [2]\tau^3v_0^4v_1$, so $[2]\tau^3v_1$ is a $v_0^4$-torsion
class.  In this fashion, we see a $v_1$-multiple of a $v_0^5$-torsion
class which is killed by $v_0^4$.

The above scenrio is a specific case of a general phenomenon in our
computations:  differentials supported
by $v_{i>0}(j)$'s (those classes at the end of dashed lines in our
MASS picture) will produce $v_0^{3+i+\nu_2(j)}$-torsion on
$[2]\tau^{2^{i+1}j-1}v_i$ while $[2]\tau^{2^{i+1}j-1}$ is $v_0^{4+i+\nu_2(j)}$-torsion.
\end{rmk}

\begin{rmk}
In the dimensional range pictured, the $E_\infty$-pages of Figures
\ref{fig:A}, \ref{fig:B}, and \ref{fig:C} almost depict
$E_\infty$ for $\pi_\star \BP_\QQ$.  The only difference is that the
$v_0^6$-torsion class in degree $15-16\alpha$ becomes a
$v_0^7$-torsion class.  This is because the class $\tau^{16}$ dies in
the $\rho$-BSS for $\BPn{n}$ when $n\ge 4$ so it is $v_0(8)$ at
height 1 that supports a $d_6$-differential killing
$[2]\tau^{15}v_0^7$ while $[2]\tau^{15}v_0^6$ survives.
\end{rmk}

\begin{rmk}
The very careful reader will note that the target of $[p]\tau$ in the
$\rho$-BSS is $(\rho^2+a_p)v_0$ when $p\equiv 3\pod{4}$, so the
$\rho$-BSS portion of Figure \ref{fig:A} is slightly misleading.
Since $\rho v_0$ dies on the $E_2$-page, no harm is done.
\end{rmk}

\begin{proof}[Proof of Theorem \ref{thm:BPnQ}]
These differentials follow from a ``least energy principle''
guaranteed by the Hasse principle:

\begin{quotation}
For $x\in E_2$ of the global MASS for $\BPn{n}$ let $r$ be the
smallest $r'$ such that some $H^{\QQ_v}(x)$ supports a
$d_{r'}$-differential.  Then $d_r x = y$ for $y$ a unique lift of $d_r H_{\BPn{n}}(x)$.
\end{quotation}

We call this a least energy principle since the global MASS has
nonzero differential on $x$ as soon as its Hasse image supports
differentials, and it is a straightforward corollary of injectivity in
(\ref{eqn:MASSHasse}).

The differentials in the Theorem follow by applying this principle to
the global $E_2$-term determined in Theorem \ref{thm:rhoBSSQ} and the
local computations of Section \ref{sec:comp}.
\end{proof}

We now set some notation so that we can express a closed form for the additive structure of $\pi_\star \BPn{n}$.  For $p\equiv 3\pod{4}$ let $A_p$ be the bigraded Abelian group with
\[\begin{aligned}
  (A_p)_{-\alpha} &= \ZZ_2,\\
  (A_p)_{-2\alpha+2r(1-\alpha)} &= \ZZ/2\text{ for }r\ge 0,\\
  (A_p)_{1-3\alpha+2r(1-\alpha)} &= \ZZ/2^{\lambda(p)-1+\nu_2(2r+2)}\text{ for }r\ge 0,\\
  (A_p)_{k+\ell\alpha} &= 0\text{ otherwise}.
\end{aligned}\]
Define $A = \bigoplus_{p\equiv 3\pod{4}} A_p$ and let
\[
  A(n) = A[v_1,\ldots,v_n].
\]

For $p\equiv 1\pod{4}$ let $B_p$ be the bigraded Abelian group with
\[\begin{aligned}
  (B_p)_{-\alpha} &= \ZZ_2,\\
  (B_p)_{-2\alpha+r(1-\alpha)} &= \ZZ/2^{\varepsilon(p)+\nu_2(r+1)}\text{ for }r\ge 0,\\
  (B_p)_{k+\ell\alpha} &= 0\text{ otherwise}.
\end{aligned}\]
Define $B = \bigoplus_{p\equiv 1\pod{4}} B_p$ and let
\[
  B(n) = B[v_1,\ldots,v_n].
\]

Now define
\[
  C'(n) = \ZZ_2[\rho,\tau^{2^{n+1}},v_1,\ldots,v_n,w_i(j) \mid 1\le i\le n, 0\le j]
\]
subject to the relations
\[\begin{aligned}
  \rho v_i &= w_i(0), \\
  \rho^{2^{i+1}-1}v_i &= 0,\\
  \rho^{2^{i+1}-2}w_i(j) &= 0,\\
  w_i(j)w_k(\ell) &= w_i(j+2^{k-i}\ell) w_k(0)\text{ if }i\le k,\\
  w_i(j) &= \tau^{2^{n+1}}w_i(j-2^{n-i})\text{ if }j\ge 2^{n-1}.
\end{aligned}\]
(Note that the class $w_i(j)$ is represented by $\rho\tau^{2^{i+1}j}v_i$ in the
MASS for $\BPn{n}$.)  Also define $C''(n)$ to be the bigraded Abelian group with
\[\begin{aligned}
  C''(n)_{-\alpha+2r(1-\alpha)} &= \ZZ_2,\\
  C''(n)_{1-2\alpha+2r(1-\alpha)} &= \ZZ/2^{3+\nu_2(2r+2)}\text{ for }r\ge 0
\text{ and }2^{n+1}\nmid 2r+2,\\
  C''(n)_{1-2\alpha+2r(1-\alpha)} &= \ZZ/2^{2+\nu_2(2r+2)}\text{ for }r\ge 0\text{ and }2^{n+1}\mid 2r+2.\\
\end{aligned}\]
Let $t_j$ denote the generator of $C''(n)$ in degree $-\alpha+j(1-\alpha)$ and define
\[
  C'''(n) = C''(n)[v_1,\ldots,v_n]/(2^{2+\nu_2(j+1)}t_jv_i \mid j\equiv 3\pod{4}, 1\le i\le \nu_2(j+1)-1).
\]
We define
\[
  C(n) = C'(n)\oplus C'''(n).
\]

Note that $A(n)$, $B(n)$, and $C(n)$ capture precisely the information on the $E_\infty$ pages presented in Figures \ref{fig:A}, \ref{fig:B}, and \ref{fig:C}, respectively, once $v_0 = 2$ is taken into account.  The following result is now a direct consequence of Theorem \ref{thm:BPnQ}.

\begin{thm}\label{thm:BPnQHtpy}
The coefficients of $\BPn{n}$ over $\QQ$ take the form
\[
  \pi_\star \BPn{n} = A(n)\oplus B(n)\oplus C(n)
\]
additively.
\end{thm}

\begin{rmk}\label{rmk:KQ}
By Lemma \ref{lem:kgl} in the $n=1$ case we can recover the
Rognes-Weibel \cite{RW} computation of the 2-complete algebraic $K$-theory of $\QQ$ by inverting $v_1$ and looking at the weight $0$ piece of the coefficients.
\end{rmk}

\begin{rmk}\label{rmk:loc}
Recall Quillen's localization fiber sequence for 2-complete algebraic
$K$-theory spectra
\[
  \bigvee_p K\FF_p \to K\ZZ \to K\QQ.
\]
The associated long exact sequence on
homotopy groups splits into isomorphisms $K_{2i}\ZZ \cong K_{2i}
\QQ$ and split short exact sequences
\[
  0\to K_{2i+1}\ZZ \to K_{2i+1}\QQ \to \bigoplus_p K_{2i}\FF_p\to 0
\]
for $i\ge 1$.  In the $n=1$ case, after inverting $v_1$ we see that
the decomposition in Theorem \ref{thm:BPnQHtpy} respects the
localization decomposition of $K_*\QQ$ in the following sense:
\[\begin{aligned}
  v_1^{-1}A(1)_m\oplus v_1^{-1}B(1)_m &\cong \bigoplus_p
  K_{m-1}\FF_p,\\
  v_1^{-1}C(1)_m &\cong K_m\ZZ
\end{aligned}\]
abstractly for $m\ge 2$.

Let $\EEE(n) = v_n^{-1}\BPn{n}$ denote the motivic Johnson-Wilson
spectrum.  (Note that $\EEE(1) = \KGL$ in the 2-complete category.)  We can ask then whether there are integral models of the
rational and $\FF_p$ truncated Brown-Peterson and Johnson-Wilson
spectra so that there are localization fiber sequences
\[
  \bigvee_p \BPn{n}_{\FF_p}^{\ZZ} \to \BPn{n}_\ZZ \to \BPn{n}_\QQ^\ZZ,
\]
\[
  \bigvee_p \EEE(n)_{\FF_p}^\ZZ \to \EEE(n)_\ZZ \to \EEE(n)_\QQ^\ZZ
\]
in the category of motivic spectra over $\Spec \ZZ$.  Joseph Ayoub
has informed us that these sorts of localization sequences should be
related to Quillen purity theorems for $\BPn{n}$ and $\EEE(n)$.  Such results
should follow from purity for algebraic $K$-theory in the $n=1$ case but $n>1$ is
\emph{terra incognita}.

On the basis of our MASS computations over $\QQ$, we might wildly speculate that
the MASS for $\BPn{n}_\ZZ$ would match portion of the MASS for
$\BPn{n}_\QQ$ presented in Figure \ref{fig:C} and the MASS for
$\BPn{n}_{\FF_p}$ would match the portion of the MASS for
$\BPn{n}_\QQ$ presented in Figure \ref{fig:A} or \ref{fig:B}
(depending on whether $p\equiv 3$ or $1\pod{4}$, respectively).  That said, essentially nothing is known about the structure of $\mathcal{A}_\star$ or convergence of the MASS over $\ZZ$.

By working in the stable
motivic homotopy category over $\FF_p$, the authors have currently verified that the MASS for
$\BPn{n}_{\FF_p}$, $p>2$, does indeed behave this way \cite{OOFp}.  (Work of Mark Hoyois, Sean Kelly, and the second author \cite{HoyKO} identifies the dual motivic Steenrod algebra at 2 over
$\FF_p$ for $p>2$, which then permits methods similar to those in Section \ref{sec:comp} of this
paper to be applied.)  The MASS for $\BPn{n}_{\FF_2}$ remains mysterious.
\end{rmk}

\bibliographystyle{plain} 
\bibliography{OO}

\begin{thebibliography}{10}

\bibitem{Bor}
S.~Borghesi.
\newblock Algebraic {M}orava {$K$}-theories.
\newblock {\em Invent. Math.}, 151(2):381--413, 2003.

\bibitem{DI}
D.~Dugger and D.~C. Isaksen.
\newblock The motivic {A}dams spectral sequence.
\newblock {\em Geom. Topol.}, 14(2):967--1014, 2010.

\bibitem{Hill}
M.~A. Hill.
\newblock Ext and the motivic {S}teenrod algebra over {$\Bbb R$}.
\newblock {\em J. Pure Appl. Algebra}, 215(5):715--727, 2011.

\bibitem{Hoyois}
M.~Hoyois.
\newblock From algebraic cobordism to motivic cohomology.
\newblock arXiv:1210.7182.

\bibitem{HoyKO}
M.~Hoyois, S.~Kelly, and P.~A. {\O}stv{\ae}r.
\newblock The motivic {S}teenrod algebra over perfect fields.
\newblock Preprint, 2012.

\bibitem{Hu}
P.~Hu.
\newblock On {R}eal-oriented {J}ohnson-{W}ilson cohomology.
\newblock {\em Algebr. Geom. Topol.}, 2:937--947, 2002.

\bibitem{HuKriz}
P.~Hu and I.~Kriz.
\newblock Some remarks on {R}eal and algebraic cobordism.
\newblock {\em $K$-Theory}, 22(4):335--366, 2001.

\bibitem{KO}
P.~Hu, I.~Kriz, and K.~Ormsby.
\newblock Convergence of the motivic {A}dams spectral sequence.
\newblock {\em J. K-Theory}, 7(3):573--596, 2011.

\bibitem{HKO}
P.~Hu, I.~Kriz, and K.~Ormsby.
\newblock Remarks on motivic homotopy theory over algebraically closed fields.
\newblock {\em J. K-Theory}, 7(1):55--89, 2011.

\bibitem{MilnorSteenrodAlg}
J.~Milnor.
\newblock The {S}teenrod algebra and its dual.
\newblock {\em Ann. of Math. (2)}, 67:150--171, 1958.

\bibitem{MilnorK}
J.~Milnor.
\newblock Algebraic {$K$}-theory and quadratic forms.
\newblock {\em Invent. Math.}, 9:318--344, 1969/1970.

\bibitem{Morelintroduction}
F.~Morel.
\newblock An introduction to {$\Bbb A^1$}-homotopy theory.
\newblock In {\em Contemporary developments in algebraic {$K$}-theory}, ICTP
  Lect. Notes, XV, pages 357--441 (electronic). Abdus Salam Int. Cent. Theoret.
  Phys., Trieste, 2004.

\bibitem{Morelstableconnectivity}
F.~Morel.
\newblock The stable {${\Bbb A}^1$}-connectivity theorems.
\newblock {\em $K$-Theory}, 35(1-2):1--68, 2005.

\bibitem{MorelVoevodsky}
F.~Morel and V.~Voevodsky.
\newblock {${\bf A}^1$}-homotopy theory of schemes.
\newblock {\em Inst. Hautes \'Etudes Sci. Publ. Math.}, (90):45--143 (2001),
  1999.

\bibitem{NSO3}
N.~Naumann, M.~Spitzweck, and P.~A. {\O}stv{\ae}r.
\newblock Existence and uniqueness of ${E}_{\infty}$-structures on motivic
  {K}-theory spectra.
\newblock arXiv:1010.3944.

\bibitem{NSO1}
N.~Naumann, M.~Spitzweck, and P.~A. {\O}stv{\ae}r.
\newblock Chern classes, ${K}$-theory and landweber exactness over nonregular
  base schemes.
\newblock In {\em Motives and Algebraic Cycles: A celebration in Honour of
  Spencer J.~Bloch}, Fields Institute Communications, Vol.~56, pages 307--317.
  AMS, Providence, RI, 2009.

\bibitem{NSO2}
N.~Naumann, M.~Spitzweck, and P.~A. {\O}stv{\ae}r.
\newblock Motivic {L}andweber exactness.
\newblock {\em Doc. Math.}, 14:551--593 (electronic), 2009.

\bibitem{NSW}
J.~Neukirch, A.~Schmidt, and K.~Wingberg.
\newblock {\em Cohomology of number fields}, volume 323 of {\em Grundlehren der
  Mathematischen Wissenschaften [Fundamental Principles of Mathematical
  Sciences]}.
\newblock Springer-Verlag, Berlin, second edition, 2008.

\bibitem{Ormsby}
K.~Ormsby.
\newblock Motivic invariants of {$p$}-adic fields.
\newblock {\em J. K-Theory}, 7(3):597--618, 2011.

\bibitem{OOFp}
K.~Ormsby and P.~A. {\O}stv{\ae}r.
\newblock Motivic invariants of low-dimensional fields.
\newblock In preparation.

\bibitem{Quillen}
D.~Quillen.
\newblock On the formal group laws of unoriented and complex cobordism theory.
\newblock {\em Bull. Amer. Math. Soc.}, 75:1293--1298, 1969.

\bibitem{Ravenelbook}
D.~C. Ravenel.
\newblock {\em Complex cobordism and stable homotopy groups of spheres}, volume
  121 of {\em Pure and Applied Mathematics}.
\newblock Academic Press Inc., Orlando, FL, 1986.

\bibitem{RW}
J.~Rognes and C.~Weibel.
\newblock Two-primary algebraic {$K$}-theory of rings of integers in number
  fields.
\newblock {\em J. Amer. Math. Soc.}, 13(1):1--54, 2000.
\newblock Appendix A by Manfred Kolster.

\bibitem{ROCRASmodules}
O.~R{\"o}ndigs and P.~A. {\O}stv{\ae}r.
\newblock Motives and modules over motivic cohomology.
\newblock {\em C. R. Math. Acad. Sci. Paris}, 342(10):751--754, 2006.

\bibitem{ROAdvancesmodules}
O.~R{\"o}ndigs and P.~A. {\O}stv{\ae}r.
\newblock Modules over motivic cohomology.
\newblock {\em Adv. Math.}, 219(2):689--727, 2008.

\bibitem{Spitzweckslices}
M.~Spitzweck.
\newblock Relations between slices and quotients of the algebraic cobordism
  spectrum.
\newblock {\em Homology, Homotopy Appl.}, 12(2):335--351, 2010.

\bibitem{Suslinlocalfields}
A.~A. Suslin.
\newblock On the {$K$}-theory of local fields.
\newblock In {\em Proceedings of the {L}uminy conference on algebraic
  {$K$}-theory ({L}uminy, 1983)}, volume~34, pages 301--318, 1984.

\bibitem{Vezzosi}
G.~Vezzosi.
\newblock Brown-{P}eterson spectra in stable {$\Bbb A^1$}-homotopy theory.
\newblock {\em Rend. Sem. Mat. Univ. Padova}, 106:47--64, 2001.

\bibitem{Voevodskystable}
V.~Voevodsky.
\newblock {$\mathbf{A}\sp 1$}-homotopy theory.
\newblock In {\em Proceedings of the International Congress of Mathematicians,
  Vol. I (Berlin, 1998)}, number Extra Vol. I, pages 579--604 (electronic),
  1998.

\bibitem{Voevodskymod2}
V.~Voevodsky.
\newblock Motivic cohomology with {${\bf Z}/2$}-coefficients.
\newblock {\em Publ. Math. Inst. Hautes \'Etudes Sci.}, (98):59--104, 2003.

\bibitem{Voevodskyreduced}
V.~Voevodsky.
\newblock Reduced power operations in motivic cohomology.
\newblock {\em Publ. Math. Inst. Hautes \'Etudes Sci.}, (98):1--57, 2003.

\bibitem{VoevodskyEMspaces}
V.~Voevodsky.
\newblock Motivic {E}ilenberg-{M}aclane spaces.
\newblock {\em Publ. Math. Inst. Hautes \'Etudes Sci.}, (112):1--99, 2010.

\bibitem{Yagita}
N.~Yagita.
\newblock Applications of {A}tiyah-{H}irzebruch spectral sequences for motivic
  cobordism.
\newblock {\em Proc. London Math. Soc. (3)}, 90(3):783--816, 2005.

\end{thebibliography}
\vspace{0.1in}

\end{document}